\documentclass{amsart}
\usepackage{amsfonts}
\usepackage{amssymb}
\usepackage{amsrefs}
\usepackage{url}
\usepackage{graphicx}
\usepackage{xcolor}
\usepackage{verbatim}
\usepackage{enumerate}
\usepackage{hyperref}
\usepackage[nameinlink]{cleveref}
\usepackage{subfig}
\usepackage{csvsimple}
\usepackage{tikz}
\usepackage[american]{circuitikz}
\usepackage{pgfplots}
\usepackage{nicefrac} 
\usepackage{mathtools}
\usepackage{mathrsfs}
\usepackage{multirow}
\usepackage{tabstackengine}
\usepackage{array}
\setstacktabbedgap{1.5ex}
\usepackage{arydshln}

\newcommand\scalemath[2]{\scalebox{#1}{\mbox{\ensuremath{\displaystyle #2}}}}
\usepackage[acronym]{glossaries}
\newacronym{res}{R}{resistor}
\newacronym{cap}{C}{capacitor}
\newacronym{ind}{L}{inductor}
\newacronym{tra}{T}{transformer}
\newacronym[firstplural=linear matrix inequalities (LMIs)]{lmi}{LMI}{linear matrix inequality}

\newcommand{\abs}[1]{\ensuremath{\left\vert#1\right\vert}}
\newcommand{\beh}{\ensuremath{\mathcal{B}_{\scalemath{.5}{\sqfunof{\begin{array}{c|c}A&B\\\hline{}C&D\end{array}}}}}}
\newcommand{\bm}[1]{\ensuremath{
\scalemath{.8}{\begin{bmatrix}#1\end{bmatrix}
}}}
\newcommand{\cfunof}[1]{\ensuremath{\left\{#1\right\}}}
\newcommand{\funof}[1]{\ensuremath{\left(#1\right)}}
\newcommand{\jw}{\ensuremath{\left(j\omega\right)}}
\newcommand{\Lloc}[1]{\ensuremath{\mathcal{L}_1^{\mathrm{loc}#1}}}
\newcommand{\norm}[1]{\ensuremath{\left\Vert #1 \right\Vert}}
\newcommand{\R}{\ensuremath{\mathbb{R}}}
\newcommand{\s}{\ensuremath{\left(s\right)}}
\newcommand{\sqfunof}[1]{\ensuremath{\left[#1\right]}}
\newcommand{\tm}{\ensuremath{\left(t\right)}}
\newcommand{\vect}[1]{\ensuremath{{#1}}}

\newcommand\zz[1]{\addstackgap[8pt]{$\scalemath{.8}{\begin{bmatrix}#1\end{bmatrix}}$}}
\newcommand\mmfull[4]{\addstackgap[8pt]{$\scalemath{.8}{
\sqfunof{\begin{array}{c;{2pt/2pt}c}
  #1&#2\\\hdashline[2pt/2pt]
  #3&#4
\end{array}
}}$}}
\newcommand\mmfulla[4]{\scalemath{.8}{
\sqfunof{\begin{array}{c;{2pt/2pt}c}
  #1&#2\\\hdashline[2pt/2pt]
  #3&#4
\end{array}
}}}
\newcommand\mmempty[1]{\scalemath{.8}{
\sqfunof{\begin{array}{c;{2pt/2pt}c}
  &\\\hdashline[2pt/2pt]
  &#1
\end{array}
}}}

\DeclareMathOperator{\trace}{tr}

\numberwithin{equation}{section}
\crefformat{equation}{(#2#1#3)}
\crefmultiformat{equation}{(#2#1#3)}%
{ and~(#2#1#3)}{, (#2#1#3)}{ and~(#2#1#3)}

\newtheorem{theorem}{Theorem}

\theoremstyle{remark}
\newtheorem{remark}{Remark}
\newtheorem{problem}{Problem}
\crefname{problem}{problem}{problems}
\newtheorem{example}{Example}

\tikzset{>=stealth}
\usetikzlibrary{calc, intersections, through, decorations.pathreplacing,decorations.markings}

\tikzset{
  on each segment/.style={
    decorate,
    decoration={
      show path construction,
      moveto code={},
      lineto code={
        \path [#1]
        (\tikzinputsegmentfirst) -- (\tikzinputsegmentlast);
      },
      curveto code={
        \path [#1] (\tikzinputsegmentfirst)
        .. controls
        (\tikzinputsegmentsupporta) and (\tikzinputsegmentsupportb)
        ..
        (\tikzinputsegmentlast);
      },
      closepath code={
        \path [#1]
        (\tikzinputsegmentfirst) -- (\tikzinputsegmentlast);
      },
    },
  },
  mid arrow/.style={postaction={decorate,decoration={
        markings,
        mark=at position .5 with {\arrow[#1]{stealth}}
      }}},
}

\pgfplotsset{compat=1.13}
\makeatother
\pgfdeclarelayer{bg}
\pgfsetlayers{bg,main}

\hypersetup{
    colorlinks = true,
    linkcolor = orange!80!black,
    anchorcolor = blue,
    citecolor = orange!80!black,
    filecolor = blue,
    urlcolor = orange!80!black,
    }

\begin{document}

\title[RLCT networks: From simple models to Simple optimal controllers]{Passive and reciprocal networks: From simple models to simple optimal controllers}
\author{Richard Pates}

\thanks{Department of Automatic Control, Lund University, Box 118, SE-221 00, Lund, Sweden. \textit{E-mail:} \href{mailto:richard.pates@control.lth.se}{\texttt{richard.pates@control.lth.se}}}

\thanks{The author is a member of the ELLIIT Strategic Research Area at Lund University. This work was supported by the ELLIIT Strategic Research Area. This project has received funding from VR 2016-04764, SSF RIT15-0091 and ERC grant agreement No 834142.}

\thanks{This work has been submitted to the IEEE for possible publication. Copyright may be transferred without notice, after which this version may no longer be accessible}


\maketitle

\begin{abstract}
Networks constructed out of resistors, inductors, capacitors and transformers form a compelling subclass of simple models. Models constructed out of these basic elements are frequently used to explain phenomena in large-scale applications, from inter-area oscillations in power systems, to the transient behaviour of optimisation algorithms. Furthermore they capture the dynamics of the most commonly applied controllers, including the PID controller. In this paper we show that the inherent structure in these networks can be used to simplify, or even solve analytically, a range of simple optimal control problems. We illustrate these results by designing and synthesising simple, scalable, and globally optimal control laws for solving constrained least squares problems, regulating electrical power systems with stochastic renewable sources, studying the robustness properties of consensus algorithms, and analysing heating networks.
\end{abstract}

\glsresetall
\section{Introduction}
Can simple systems be regulated by simple controllers? Engineering experience indicates that this is often the case. Certainly the overwhelming majority of industrially deployed controllers are PID controllers, many of which are tuned without detailed models \cite{AH01,DM02}. This is often in disappointing contrast with the controllers produced using model based approaches. For example, whilst the properties of $H_2$ and $H_\infty$ methods are often highly desirable in applications, the controllers they produce are often more complex than desired. This is particularly damaging in large-scale applications, such as the control of electrical power systems, where simplicity and scalability are of paramount importance, yet features such as sparsity are notoriously difficult to design for in an optimal fashion \cite{Wit68}.

These shortcomings, coupled with the increasing importance of large-scale problems, has led to a great deal of research into methods for optimising performance under structural constraints. It is not possible to do justice to the rich literature along this line. However some notable themes include the emergence of Quadratic Invariance as a central concept in structured synthesis \cite{RL06}, the use of tools from the theory of large-scale optimisation (for example \cite{LFJ13}), and the exploitation of alternative controller parametrisations that maintain convexity under sparsity constraints \cite{WMD18}. In this note we take a slightly different approach. Rather than imposing structural constraints on the controller (which is difficult and diminishes the level of achievable performance), we instead look for structural features in system models that naturally yield simpler optimal controllers. This is similar in philosophy to the work of \cite{BPD02}, in which the property of spatial invariance is shown to simplify the synthesis and realisation of optimal controllers for a class of optimal control problems. However, rather than using spatial invariance to define our notion of a `simple system', we instead turn to passivity and reciprocity.

Passivity and reciprocity are compelling surrogates for simplicity. On the one hand they are baked into the physical world. The passive and reciprocal systems (under the appropriate definitions \cite{Hug19}) correspond precisely to the systems that can be constructed out of networks of resistors (R), inductors (L), capacitors (C) and transformers (T). These basic elements have analogues in many other domains (for example in both translational and rotational mechanics, thermodynamics and hydraulics \cite{SMR67}), and models constructed out of these basic elements are frequently used to explain phenomena in large-scale applications. For example inter-area oscillations in power systems are often understood by analogy with lightly damped mechanical networks \cite{Kun94}. On the other hand, passivity also forms the backbone of some of the most applicable controller structures and design methods. Perhaps most notably in this regard are PID controllers (with non-negative proportional, integral and derivative gains), which are both passive and reciprocal. Furthermore passive and reciprocal controllers can be implemented without an energy source \cite{FLN16}, and passivity based design is one of the most scalable and applicable large-scale design methods \cite{OSC08}. 

In this paper we explore the simplifications that arise in a range of $H_2$ and $H_\infty{}$ optimal control problems, when the process to be controlled can be modelled by a linear RLCT network. Of particular importance for our work are the signature symmetric passive realisations \cite{Wil72,YT66,Fur83}. These realisations are the central focus of Part II of Willems' celebrated '72 paper, where it is shown that if a passive and reciprocal network has dynamics described by a proper transfer function, then it has a controllable and observable state-space realisation
\[
\tfrac{d}{dt}x=Ax+Bu,\,y=Cx+Du,
\]
where in addition the matrices $A,B,C$ and $D$ satisfy \cite[Theorem 7]{Wil72}
\begin{equation}\label{eq:introst}
\bm{\Sigma_{\mathrm{int}}&0\\0&\Sigma_{{\mathrm{ext}}}}\bm{-A&-B\\C&D}=\bm{-A&-B\\C&D}^{\mathsf{T}}\bm{\Sigma_{\mathrm{int}}&0\\0&\Sigma_{\mathrm{ext}}},\;\bm{-A&-B\\C&D}+\bm{-A&-B\\C&D}^{\mathsf{T}}\succeq{}0.
\end{equation}
In the above $\Sigma_{\mathrm{int}}$ and $\Sigma_{\mathrm{ext}}$ are signature matrices (diagonal matrices with entries equal to $\pm{}1$), often referred to as the internal and external signatures respectively. As we shall see, the structure of $\Sigma_{\mathrm{int}}$ closely reflects the internal structure of an RLCT network (in particular the numbers of capacitors and inductors used in the network), and in special cases (such as LCT networks) the structure in \cref{eq:introst} simplifies further. From the perspective of optimal control problems, the power of this viewpoint is that it allows us to think of our useful modelling class purely in terms of structured matrices. This can be used to simplify, or even solve analytically, the Riccati equations and \glspl{lmi} arising in common optimal control problems. Furthermore it can provide additional insights into the synthesis of optimal controllers with distributed structures that are not captured by conventional notions such as sparsity. 

There is however a catch. Even deceptively simple RLC networks, such as the Bott-Duffin network in \Cref{fig:bottduffin}, may not be controllable \cite{CWB03,Wil04a,Hug17b}. One consequence of this is that we cannot deduce the existence of a structured realisation in the form of \cref{eq:introst} for these networks using classical results \cite{Wil72,YT66,Fur83,AV06}. This muddies the waters considerably. We would like to abstract our study of simple systems into the study of the algebra of structured matrices, but seemingly there is a gap between \cref{eq:introst} and our objects of interest. This bug, which strikes at the very heart of systems theory, was one of the driving forces behind another of Willems' major contributions, namely the conception of the behavioral approach \cite{PW98}. Critically for our purposes, as shown in \cite{Hug19} through the use of behavioral methods, RLCT networks are equivalently described by state-space models as structured in \cref{eq:introst}. That is, every state-space model as structured by \cref{eq:introst}, without any assumptions on observability and controllability of \funof{C,A} and \funof{A,B}, corresponds to the dynamical model of an RLCT network, and vice versa.

The rest of this paper is structured as follows. We begin by reviewing RLCT networks. The objective is to give both a brief introduction to the modelling of electrical networks, and also a complete characterisation of the additional structural features in networks built only out of subsets of the RLCT elements. The main result of this section is \Cref{tab:1}, which gives a characterisations akin to \cref{eq:introst} for these networks (\Cref{thm:2}). We also show how to extend \cref{eq:introst} so that the internal behavior of RLCT networks can also be algebraically characterised (\Cref{thm:1}).

We then turn to optimal control problems. We first show that the structure in \cref{eq:introst} typically allows the computational burden in standard $H_2$ and $H_\infty$ problems to be halved. We then show that for the subclasses of the RLCT networks considered in \Cref{tab:1}, several natural optimal control problems have analytical solutions with desirable controller implementations. We illustrate these results by designing and synthesising simple, scalable, and globally optimal control laws for solving constrained least squares problems, regulating electrical power systems with stochastic renewable sources, studying the robustness properties of consensus algorithms, and analysing heating networks.

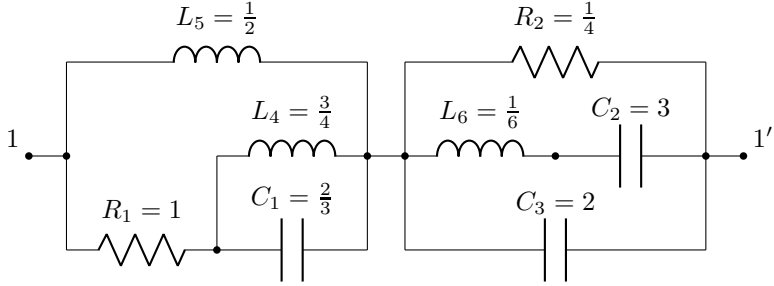
\begin{figure}
\centering
\begin{tikzpicture}[font=\normalsize]

  \draw (-0.5,-0.75) -- (0,-0.75);
  \draw (0,-2) -- (0,0.5);
  \draw (2,-2) -- (2,-0.75);
  \draw (4,-2) -- (4,0.5);
  \draw (4,-0.75) -- (4.5,-0.75);
  \draw (4.5,-2) -- (4.5,0.5);
  \draw (8.5,-2) -- (8.5,0.5);
  \draw (8.5,-0.75) -- (9,-0.75);

  \fill[black] (-0.5,-0.75) circle (0.05) node[above left] {$1$};
  \fill[black] (2,-2) circle (0.05);
  \fill[black] (9,-0.75) circle (0.05) node[above right] {$1'$};
  \fill[black] (4,-0.75) circle (0.05);
  \fill[black] (4.5,-0.75) circle (0.05);
  \fill[black] (6.5,-0.75) circle (0.05);
  \fill[black] (0,-0.75) circle (0.05);
  \fill[black] (8.5,-0.75) circle (0.05);

  \draw (0,0.5) to [L] (4,0.5);
  \node[above, yshift=0.25cm] at (2,0.5) {$L_5 = \tfrac{1}{2}$};

  \draw (2,-0.75) to [L] (4,-0.75);
  \node[above, yshift=0.25cm] at (3,-0.75) {$L_4 = \tfrac{3}{4}$};

  \draw (4.5,-0.75) to [L] (6.5,-0.75);
  \node[above, yshift=0.25cm] at (5.5,-0.75) {$L_6 = \tfrac{1}{6}$};

  \draw (0,-2) to [R] (2,-2);
  \node[above, yshift=0.25cm] at (1,-2) {$R_1 = 1$};

  \draw (4.5,0.5) to [R] (8.5,0.5);
  \node[above, yshift=0.25cm] at (6.5,0.5) {$R_2 = \tfrac{1}{4}$};

  \draw (2,-2) to [C] (4,-2);
  \node[above, yshift=0.35cm] at (3,-2) {$C_1 = \tfrac{2}{3}$};

  \draw (6.5,-0.75) to [C] (8.5,-0.75);
  \node[above, yshift=0.35cm] at (7.5,-0.75) {$C_2 = 3$};

  \draw (4.5,-2) to [C] (8.5,-2);
  \node[above, yshift=0.35cm] at (6.5,-2) {$C_3 = 2$};

\end{tikzpicture}
\caption[]{\label{fig:bottduffin}Example of an electrical network with an uncontrollable behavior. This example is taken from \cite{Hug19} (and fixes a minor typo). The behavior of this network does not admit a controllable and observable state-space realisation. However it does admit a signature symmetric realisation in the form of \cref{eq:introst}:\vspace{.5cm}
\newline{}
\begin{minipage}{\linewidth}
     $$
     \begin{aligned}
        \tfrac{d}{dt}x&=\bm{-2&0&0&-\sqrt{3}&0&0\\
        0&0&0&-\sqrt{2}&0&0\\
        0&0&0&0&\sqrt{2}&-\sqrt{3}\\
        \sqrt{3}&\sqrt{2}&0&0&0&0\\
        0&0&-\sqrt{2}&0&0&0\\
        0&0&\sqrt{3}&0&0&-2}x+\bm{-\sqrt{2}\\0\\0\\\sqrt{3/2}\\0\\1/\sqrt{2}}i,\\
        v&=\bm{\sqrt{2}&0&0&\sqrt{3/2}&0&1/\sqrt{2}}x+i.
      \end{aligned}
     $$
  \end{minipage}
\vspace{.2cm}\newline{}
\noindent{}
In the above $\funof{i,v}$ denote the current through, and the voltage across, the terminal pair \funof{1,1'}.
}
\end{figure}

\section{Notation}

$M\succ{}0$ denotes that a matrix $M$ is symmetric and positive definite, and $M\succeq{}0$ that it is symmetric and positive semi-definite. The unique positive semi-definite square root of $M\succeq{}0$ is denoted $M^{\frac{1}{2}}$. The trace of a matrix $M$ is denoted by $\trace{M}$. A signature matrix is a diagonal matrix with diagonal entries equal to $\pm1$. A permutation matrix is a square matrix in which every row and column has exactly one non-zero entry equal to 1. $I$ denotes the identity matrix throughout, and $0$ a matrix of zeros. A pair of matrices $\funof{A,B}$, where $A\in\R^{n\times{}n}$, is said to be controllable if the matrix
\[
\bm{B&AB&\cdots{}&A^{n-1}B}
\]
has full rank. The pair $\funof{A,C}$ is said to be observable if $\funof{A^{\mathsf{T}},C^{\mathsf{T}}}$ is controllable. $\Lloc{,n}$ denotes the \emph{n}-vector valued locally integrable functions. 

Given a state-space model
\[
\tfrac{d}{dt}x=Ax+Bu,\;y=Cx+Du,
\]
the transfer function from $u$ to $y$ is defined to be $T_{yu}\s=C\funof{sI-A}^{-1}B+D$. Whenever $T_{yu}$ is stable (and $D=0$ in the $H_2$ case), the $H_2$ norm and $H_\infty$ norm of $T_{yu}$ are defined to equal
\[
\norm{T_{yu}}_{H_2}=\funof{\frac{1}{2\pi}\int_{-\infty{}}^\infty{}\trace\funof{T_{yu}\jw^*T_{yu}\jw}\,d\omega}^{\frac{1}{2}}
\]
and
\[
\norm{T_{yu}}_{H_\infty}=\sup_{\omega\in\R}\norm{T_{yu}\jw}_2,
\]
where $\norm{\cdot}_2$ denotes the induced matrix 2-norm. 

Wherever possible uppercase letters are used to represent matrices with elements in $\R$, and lowercase letters vectors of signals with elements in $\Lloc{}$. In general the notation $M_{lk}$ and $x_k$ is \emph{not} used to index elements of matrices and vectors (i.e. $M_{lk}$ and $x_k$ do not refer to the \emph{lk}-th and \emph{k}-th elements of $M$ and $x$ respectively); instead indices on vectors and matrices will be used to refer to particular sub-blocks, though the meaning of any indexing will normally be explicitly defined. On occasion, dashed lines will be introduced to clarify dimensional compatibility in equations involving block matrices. For example, the equation
\[
\mmfulla{A}{B}{C}{0}=
\scalemath{.8}{
\sqfunof{\begin{array}{cc;{2pt/2pt}c}
  0 & X & 0\\
  -X^{\mathsf{T}} & 0 & Y\\\hdashline[2pt/2pt]
  0 & Y^{\mathsf{T}} & 0
\end{array}
}}
\quad\Rightarrow{}\quad{}A=\bm{0 &  X\\ -X^{\mathsf{T}} & 0}\;\;\text{and}\;\;B=C^{\mathsf{T}}=\bm{0\\Y}.
\]

\section{Network Behaviors with Internally Symmetric Realisations}

The objective of this section is to show how to describe the dynamics of linear RLCT networks using structured state-space models. Establishing this connection will form the central ingredient in our results on the optimal control of systems that can be modelled using RLCT networks. As we will see, despite their conceptual simplicity, translating the structural features of the dynamical models of RLCT networks into the state-space setting is not entirely straightforward, and will require quite a heavy notational burden. However it should be emphasised that this notation will not be required in the subsequent section on the design of optimal controllers, where only the structure in the $A$, $B$, $C$ and $D$ matrices in the state-space descriptions will be used. The reader content with their understanding of electrical networks can freely skip ahead after consulting \Cref{tab:1}, and a more leisurely (and more complete) introduction to the covered topics can be found in \cite{Hug17,AV06}.

\subsection{Preliminaries on Modelling RLCT Networks}

\begin{figure}

\ctikzset{bipoles/length=.56cm}
\ctikzset{bipoles/thickness=1}
\centering
\begin{tikzpicture}[font=\normalsize,>=stealth]

  \draw (4,11) -- (0,11);
  \draw (4,10) -- (3,10);
  \draw (4,9) -- (0,9);
  \draw (4,8) -- (3,8);
  \draw (4,7) -- (1,7);
  \draw (4,6) -- (3,6);
  \draw (4,5) -- (2,5);
  \draw (4,4) -- (3.5,4);
  \draw (4,3) -- (3.5,3);
  \draw (4,2) -- (1,2);
  \draw (4,1) -- (3,1);
  \draw (4,0) -- (2,0);

  \draw (0,11) -- (0,9);
  \draw (1,7) -- (1,6);
  \draw (1,5) -- (1,2);
  \draw (2,2) -- (2,0);
  \draw (3.5,4) -- (3.5,3);

  \draw (0,9) to [R] node[yshift=0.85cm, xshift=0.5cm] {$R_1$} (0,6);
  \draw (2,5) to [R] node[yshift=.85cm, xshift=0.5cm] {$R_2$} (2,2);

  \draw (-0.5,6) -- (1,6);
  \draw (-0.5,5) -- (1,5);
  \fill[black] (-0.5,6) circle (0.05);
  \fill[black] (-0.5,5) circle (0.05);

  \draw[->] (-0.5,5.1) -- node[left] {$v_{\mathrm{ext}}$} (-0.5,5.9);
  \draw[->] (-1.1,6) node[left] {$i_{\mathrm{ext}}$} -- (-0.6,6);
  \draw[<-] (-1.1,5) -- (-0.6,5);

  \draw (3,10) -- (3,9.1);
  \draw (3,8.9) -- (3,7.1);
  \draw (3,6.9) -- (3,2.1);
  \draw (3,1.9) -- (3,1);

  \draw (3,8.9) arc[start angle=-90, end angle=90, radius=.1];
  \draw (3,6.9) arc[start angle=-90, end angle=90, radius=.1];
  \draw (3,1.9) arc[start angle=-90, end angle=90, radius=.1];

  \fill[black] (0,11) circle (0.05);
  \fill[black] (0,9) circle (0.05);
  \fill[black] (0,6) circle (0.05);
  \fill[black] (2,5) circle (0.05);
  \fill[black] (2,2) circle (0.05);
  \fill[black] (3,8) circle (0.05);
  \fill[black] (3,6) circle (0.05);
  \fill[black] (3,5) circle (0.05);

  \fill[black] (4,11) circle (0.05);
  \fill[black] (4,10) circle (0.05);
  \fill[black] (4,9) circle (0.05);
  \fill[black] (4,8) circle (0.05);
  \fill[black] (4,7) circle (0.05);
  \fill[black] (4,6) circle (0.05);
  \fill[black] (4,5) circle (0.05);
  \fill[black] (4,4) circle (0.05);
  \fill[black] (4,3) circle (0.05);
  \fill[black] (4,2) circle (0.05);
  \fill[black] (4,1) circle (0.05);
  \fill[black] (4,0) circle (0.05);

  \draw[orange!80!black] (5.2,11) to [C, color=orange!80!black] node[yshift=0.27cm, xshift=0.25cm,
right] {$C_1$} (5.2,10);
  \fill[orange!80!black] (5.2,10) circle (0.05);
  \fill[orange!80!black] (5.2,11) circle (0.05);
  \draw[->] (4,10.1) -- node[right] {$v_{\mathrm{int},1}$} (4,10.9);
  \draw[->] (4.6,11) node[above] {$i_{\mathrm{int},1}$} -- (4.1,11);
  \draw[<-] (4.6,10) -- (4.1,10);

  \draw[orange!80!black] (5.2,9) to [L, color=orange!80!black] node[yshift=0.3cm, xshift=0.25cm,
right] {$L_4$} (5.2,8);
  \fill[orange!80!black] (5.2,8) circle (0.05);
  \fill[orange!80!black] (5.2,9) circle (0.05);
  \draw[->] (4,8.1) -- node[right] {$v_{\mathrm{int},4}$} (4,8.9);
  \draw[->] (4.6,9) node[above] {$i_{\mathrm{int},4}$} -- (4.1,9);
  \draw[<-] (4.6,8) -- (4.1,8);

  \draw[orange!80!black] (5.2,7) to [L, color=orange!80!black] node[yshift=0.3cm, xshift=0.25cm,
right] {$L_5$} (5.2,6);
  \fill[orange!80!black] (5.2,6) circle (0.05);
  \fill[orange!80!black] (5.2,7) circle (0.05);
  \draw[->] (4,6.1) -- node[right] {$v_{\mathrm{int},5}$} (4,6.9);
  \draw[->] (4.6,7) node[above] {$i_{\mathrm{int},5}$} -- (4.1,7);
  \draw[<-] (4.6,6) -- (4.1,6);

  \draw[orange!80!black] (5.2,5) to [L, color=orange!80!black] node[yshift=0.3cm, xshift=0.25cm,
right] {$L_6$} (5.2,4);
  \fill[orange!80!black] (5.2,4) circle (0.05);
  \fill[orange!80!black] (5.2,5) circle (0.05);
  \draw[->] (4,4.1) -- node[right] {$v_{\mathrm{int},6}$} (4,4.9);
  \draw[->] (4.6,5) node[above] {$i_{\mathrm{int},6}$} -- (4.1,5);
  \draw[<-] (4.6,4) -- (4.1,4);

  \draw[orange!80!black] (5.2,3) to [C, color=orange!80!black] node[yshift=0.27cm, xshift=0.25cm,
right] {$C_2$} (5.2,2);
  \fill[orange!80!black] (5.2,2) circle (0.05);
  \fill[orange!80!black] (5.2,3) circle (0.05);
  \draw[->] (4,2.1) -- node[right] {$v_{\mathrm{int},2}$} (4,2.9);
  \draw[->] (4.6,3) node[above] {$i_{\mathrm{int},2}$} -- (4.1,3);
  \draw[<-] (4.6,2) -- (4.1,2);

  \draw[orange!80!black] (5.2,1) to [C, color=orange!80!black] node[yshift=0.27cm, xshift=0.25cm,
right] {$C_3$} (5.2,0);
  \fill[orange!80!black] (5.2,0) circle (0.05);
  \fill[orange!80!black] (5.2,1) circle (0.05);
  \draw[->] (4,0.1) -- node[right] {$v_{\mathrm{int},3}$} (4,0.9);
  \draw[->] (4.6,1) node[above] {$i_{\mathrm{int},3}$} -- (4.1,1);
  \draw[<-] (4.6,0) -- (4.1,0);

  \fill[black,opacity=0.1] (-0.25,-0.25) rectangle (3.75,11.25);

\end{tikzpicture}
\caption{\label{fig:reactex}Illustration of the reactance extraction representation for the network from \Cref{fig:bottduffin}. First all the wires, resistors and transformers are collected into a box (the grey region). The inner workings of this box are accessible through two sets of terminal pairs. The internal terminals are drawn on the right. These are associated with a set of internal currents $i_{\mathrm{int}}$ and voltages $v_{\mathrm{int}}$, and the reactive elements (the inductors and capacitors) are connected across these terminal pairs, by connecting the orange elements.
The second set of terminal pairs on the left are associated with the external driving point currents and voltages as shown. It is through the driving points that the electrical network interacts with the outside world. Note that in this example there is only one pair of external terminals, but in general there can be any number. The indexing of the internal currents and voltages has been chosen to be consistent with \cref{eq:beh1}.}
\end{figure}
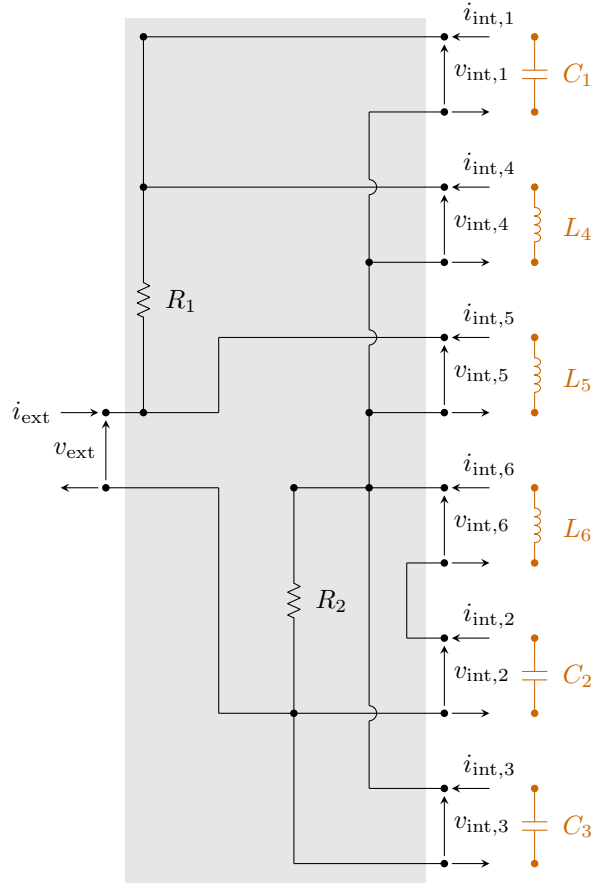

In this paper we take a reactance extraction approach to the analysis and synthesis of RLCT networks, as illustrated in \Cref{fig:reactex}. In this setting, an electrical network consists of three main parts:
\begin{enumerate}
  \item A box containing the interconnection of wires, resistors and transformers. The inner workings of this box are accessible through two sets of terminal pairs.
  \item The reactive elements (capacitors and inductors). Each reactive element is connected across a single terminal pair in the first set of terminal pairs.
  \item The driving points. This is really just another name for the second set of terminal pairs. It is through the driving points that the electrical network interacts with the outside world.
\end{enumerate}
In order to describe the dynamics of an RLCT network, we first associate each of the terminal pairs with a through current, and an across voltage. We call the currents and voltages associated with the reactive elements the internal currents and voltages, and denote them with the vectors of signals $i_{\mathrm{int}}\tm\in\Lloc{,\funof{n_{\mathrm{C}}+n_{\mathrm{L}}}}$ and $v_{\mathrm{int}}\tm\in\Lloc{,\funof{n_{\mathrm{C}}+n_{\mathrm{L}}}}$ respectively. Here $n_{\mathrm{C}}$ denotes the number of capacitors, and $n_{\mathrm{L}}$ the number of inductors. Similarly we call the currents and voltages associated with the driving points the external currents and voltages, and denote them with $i_{\mathrm{ext}}\tm\in\Lloc{,n_{\mathrm{ext}}}$ and $v_{\mathrm{ext}}\tm\in\Lloc{,n_{\mathrm{ext}}}$, where $n_{\mathrm{ext}}$ denotes the number of driving points. The mathematical model of the network is then given by the element laws for the reactive elements, and the constraints imposed by the box, which can all be written in terms of these variables. More specifically, these equations always take the form
\begin{equation}\label{eq:beh1}
\begin{aligned}
&\begin{cases}
C_k\tfrac{d}{dt}v_{\mathrm{int},k}=i_{\mathrm{int},k}&\text{if $k\in\cfunof{1,\ldots{},n_{\mathrm{C}}}$,}\\
L_k\tfrac{d}{dt}i_{\mathrm{int},k}=v_{\mathrm{int},k}&\text{if $k\in\cfunof{n_{\mathrm{C}}+1,\ldots{},n_{\mathrm{L}}+n_{\mathrm{C}}}$,}\\
\end{cases}\\
&\qquad\qquad{}X\bm{i_{\mathrm{int}}\\i_{\mathrm{ext}}}=Y\bm{v_{\mathrm{int}}\\v_{\mathrm{ext}}},
\end{aligned}
\end{equation}
where $i_{\mathrm{int},k}$ and $v_{\mathrm{int},k}$ denote the \emph{k}-th elements of $i_{\mathrm{int}}$ and $v_{\mathrm{int}}$ respectively. The first equation in the above specifies the element laws for the capacitors and inductors, where it is assumed without loss of generality that the first $n_{\mathrm{C}}$ internal terminal pairs have a capacitor connected (with inductors connected across the remaining internal terminal pairs). The second equation specifies the constraints imposed by the box of wires, resistors and transformers. In particular $X,Y\in\R^{n\times{}n}$, where $n=n_{\mathrm{C}}+n_{\mathrm{L}}+n_{\mathrm{ext}}$, are square matrices, with entries determined by the element laws for the resistors and transformers, the network topology and Kirchhoff's laws. This high level description of these constraints is sufficient to describe our main results, so we will not give further details here. However more information about the specific structure in these matrices can be found in the proof of \Cref{thm:1} (see in particular \cref{eq:p1e2}), and a complete description can also be found in \cite[Theorem 4]{Hug17}. Note also that the through currents and across voltages associated with the resistor and transformer elements do not appear explicitly in this representation. For details on how these relate to $i_{\mathrm{int}},i_{\mathrm{ext}},v_{\mathrm{int}}$ and $v_{\mathrm{ext}}$, see \cite[\S{5}]{Hug17}.

\subsection{State-space Models Descriptions of RLCT Network Behaviors}

\begin{table}
\centering

\caption{\label{tab:1}Structure of the realisations for networks constructed out of resistors (R), inductors (L), capacitors (C) and transformers (T), as guaranteed by \Cref{thm:2}. The first column gives the types of elements used (for example {\bf{RT}} indicates that the network is constructed using only resistors and transformers). The second column gives the internal signature of the realisation. The third and fourth columns give the structural properties in the matrices $A,B,C$ and $D$ that describe the behavior of the network constructed out of the given types of element.
}

\begin{tabular}{cc|c|c|c}
\hline\hline
& & $\mathbf{\Sigma_{\mathrm{int}}}$ & $\mathbf{\mmfull{A}{B}{C}{D}}$ & \textbf{Conditions} \\ \hline\hline
\multirow{4}{*}{\rotatebox{90}{\textbf{Lossless Networks}$\qquad\qquad$}} & \textbf{T}   & N/A  &
   \addstackgap[8pt]{$\scalemath{.8}{
  \sqfunof{\begin{array}{c;{2pt/2pt}cc}
    \phantom{0} &  & \\\hdashline[2pt/2pt]
     & 0 & D_{12} \\
     & -D_{12}^{\mathsf{T}} & 0
  \end{array}
  }}$} &
   -- \\ \cline{2-5}
& \textbf{LT}  & $-I$ &
    
  \addstackgap[8pt]{$\scalemath{.8}{
  \sqfunof{\begin{array}{c;{2pt/2pt}cc}
    0 & B_1 & 0\\\hdashline[2pt/2pt]
    B_1^{\mathsf{T}} & 0 & D_{12} \\
    0 & -D_{12}^{\mathsf{T}} & 0
  \end{array}
  }}$}
 &
   -- \\ \cline{2-5}
& \textbf{CT}  & $I$  &
  \addstackgap[8pt]{$\scalemath{.8}{
  \sqfunof{\begin{array}{c;{2pt/2pt}cc}
    0 & 0 & B_2\\\hdashline[2pt/2pt]
    0 & 0 & D_{12}\\
    B_2^{\mathsf{T}} & -D_{12}^{\mathsf{T}} & 0
  \end{array}
  }}$}
&
   --\\ \cline{2-5}
& \textbf{LCT} &
   $\zz{I & 0 \\ 0 & -I}$ &

   \addstackgap[8pt]{$\scalemath{.8}{
    \sqfunof{\begin{array}{cc;{2pt/2pt}cc}
      0 & A_{12} & 0 & B_{12}\\
      -A_{12}^{\mathsf{T}} & 0 & B_{21} & 0\\\hdashline[2pt/2pt]
      0 & B_{21}^{\mathsf{T}} & 0 & D_{12} \\
      B_{12}^{\mathsf{T}} & 0 & -D_{12}^{\mathsf{T}} & 0
    \end{array}
    }}$}
 &
   -- \\ \hline\hline
\multirow{3}{*}{\rotatebox{90}{\textbf{Lossy Networks}$\quad\quad$}} & \textbf{RT}  & N/A &
   \addstackgap[8pt]{$\scalemath{.8}{
  \sqfunof{\begin{array}{c;{2pt/2pt}cc}
    \phantom{0} &  & \\\hdashline[2pt/2pt]
     & D_{11} & D_{12} \\
     & -D_{12}^{\mathsf{T}} & D_{22}
  \end{array}
  }}$} &
   $D_{11} \succeq 0, D_{22} \succeq 0$\\ \cline{2-5}
& \textbf{RLT} & $-I$ &
   \addstackgap[8pt]{$\scalemath{.8}{
  \sqfunof{\begin{array}{c;{2pt/2pt}cc}
    A & B_1 & B_2 \\\hdashline[2pt/2pt]
    B_1^{\mathsf{T}} & D_{11} & D_{12} \\
    -B_2^{\mathsf{T}} & -D_{12}^{\mathsf{T}} & D_{22}
  \end{array}
  }}$} &
   $\zz{-A & B_2 \\ B_2^{\mathsf{T}} & D_{22}} \succeq 0, D_{11} \succeq 0$\\ \cline{2-5}
& \textbf{RCT} &
   $I$ &
   \addstackgap[8pt]{$\scalemath{.8}{
  \sqfunof{\begin{array}{c;{2pt/2pt}cc}
    A & B_1 & B_2 \\\hdashline[2pt/2pt]
    -B_1^{\mathsf{T}} & D_{11} & D_{12} \\
    B_2^{\mathsf{T}} & -D_{12}^{\mathsf{T}} & D_{22}
  \end{array}
  }}$} &
   $\zz{-A & B_1 \\ B_1^{\mathsf{T}} & D_{11}} \succeq 0, D_{22} \succeq 0$\\ \hline\hline

\end{tabular}

\end{table}

In this subsection we will give a state-space characterisation of the models of RLCT networks, that includes all the external and internal currents and voltages in \cref{eq:beh1}. Our objective is to do this through a set of algebraic constraints on the matrices that define the state-space models, along the lines of \cref{eq:introst}. This will be our first step towards the structured state-space descriptions in \Cref{tab:1}. However in order to capture the internal structure in these networks, we will need to introduce several additional matrices beyond those in \cref{eq:introst}. We will see in the next subsection that if only the external currents and voltages are to be described (which is all we require when studying optimal control problems), only one of these additional matrices is needed, and a representation that is far closer to \cref{eq:introst} can be given.

The model in \cref{eq:beh1}, just like a state-space model, is just a set of differential and algebraic equations. We can therefore productively think about this model in terms of the solutions to these equations. This is precisely the behavioral perspective. The set of locally integrable functions $\Lloc{,n}$ provides a rich enough solution set, and so we say that the behavior of an RLCT network is equal to the set of locally integrable functions that solve \cref{eq:beh1} (where differentiation is interpreted in the weak sense). This set takes the general form
\begin{equation}\label{eq:beh11}
\mathcal{B}=\cfunof{\funof{i,v}\in\Lloc{,n}\times{}\Lloc{,n}:G\funof{\tfrac{d}{dt}}i=H\funof{\tfrac{d}{dt}}v},
\end{equation}
where $G$ and $H$ are matrices of polynomials. The entries of these matrices are determined by \cref{eq:beh1}, and $i$ and $v$ are the vectors of external and internal currents, and external and internal voltages
\[
i=\bm{i_{\mathrm{int}}\\i_{\mathrm{ext}}},\;v=\bm{v_{\mathrm{int}}\\v_{\mathrm{ext}}}.
\]
State-space models can be treated in a similar fashion. We will have repeated use for this, and so introduce the shorthand 
\[
\beh=\cfunof{\funof{x,u,y}\in\Lloc{,p}\times{}\Lloc{,n}\times{}\Lloc{,m}:\tfrac{d}{dt}x=Ax+Bu,y=Cx+Du}
\]
to denote the behavior of a state-space model defined by matrices $A,B,C$ and $D$. Our objective is now, given a behavior $\mathcal{B}$ of an RLCT network, to find a behavior defined by a state-space model
\begin{equation}\label{eq:ssbeh}
\tilde{\mathcal{B}}=\cfunof{\funof{i,v}:\funof{x,u,y}\in\beh{},\bm{i\\v}=F\bm{x\\u\\y}},
\end{equation}
such that $\mathcal{B}=\mathcal{\tilde{B}}$ (up to a set of measure zero). In the above the matrix $F$ specifies a linear transformation between the state-space variables and the external and internal current and voltage variables in the RLCT network. Our intention is to identify the structural features that must be present in the matrices $A,B,C,D$ and $F$ so that every RLCT network corresponds to such a structured state-space model, and vice versa.

The following theorem shows that the full behavior of the RLCT networks with $n_{\mathrm{ext}}$ driving points, $n_{\mathrm{C}}$ capacitors, and $n_{\mathrm{L}}$ inductors, can be characterised through \cref{eq:ssbeh} in terms of a set of algebraic conditions on six matrices: $P\in\R^{n\times{}n},\Sigma_{\mathrm{int}}^\dagger{}\in\R^{n_{\mathrm{int}}^\dagger\times{}n_{\mathrm{int}}^\dagger{}},\Sigma_{\mathrm{int}}\in\R^{n_{\mathrm{int}}\times{}n_{\mathrm{int}}},\Sigma_{\mathrm{ext}}\in\R^{n_{\mathrm{ext}}\times{}n_{\mathrm{ext}}},$
\[
\bm{-A&-B\\C&D}\in\R^{\funof{n_{\mathrm{int}}+n_{\mathrm{ext}}}\times{}\funof{n_{\mathrm{int}}+n_{\mathrm{ext}}}}\text{ and }\bm{\Theta&\Gamma^{\mathsf{T}}\\\Gamma&\Phi}\in\R^{\funof{n_{\mathrm{int}}^\dagger+n_{\mathrm{int}}}\times{}\funof{n_{\mathrm{int}}^\dagger+n_{\mathrm{int}}}},
\]
where $n_{\mathrm{int}}^\dagger{}+n_{\mathrm{int}}=n_{\mathrm{C}}+n_{\mathrm{L}}$ and $n=n_{\mathrm{C}}+n_{\mathrm{L}}+n_{\mathrm{ext}}$. These are the conditions \textit{(2)a)--d)} in the theorem statement. The significance of this result is that it shows that when discussing the RLCT networks, we may equivalently reason in terms these six structured matrices. The result depends on a function $\mathcal{F}\funof{\cdot{}}$. The details of this function are unimportant for the rest of this paper and so it has been placed in \cref{eq:monster} below the theorem statement due to its alarming size.

\begin{theorem}\label{thm:1}
Let
\[
\mathcal{B}=\cfunof{\funof{\vect{i},\vect{v}}\in\Lloc{,n}\times{}\Lloc{,n}:G\funof{\tfrac{d}{dt}}\vect{i}=H\funof{\tfrac{d}{dt}}\vect{v}},
\]
where $G$ and $H$ are matrices of polynomials. Given any non-negative integers $n_{\mathrm{L}},n_{\mathrm{C}},n_{\mathrm{ext}}$ such that $n_{\mathrm{L}}+n_{\mathrm{C}}+n_{\mathrm{ext}}=n$, the following are equivalent:
\begin{enumerate}
  \item $\mathcal{B}$ is the behavior of an RLCT network constructed with $n_{\mathrm{L}}$ inductors, $n_{\mathrm{C}}$ capacitors and $n_{\mathrm{ext}}$ driving points.
  \item 
  \[
  \mathcal{B}=\cfunof{\funof{\vect{i},\vect{v}}:
  \funof{\vect{x},\vect{u},\vect{y}}\in{}\beh,
  \bm{\vect{i}\\\vect{v}}=F
    \bm{\vect{x}\\\vect{u}\\\vect{y}}
    },
  \]
  where
  \[
    F=\mathcal{F}\funof{\bm{-A&-B\\C&D},\bm{\Theta&\Gamma{}\\\Gamma^{\mathsf{T}}&\Phi},\bm{-\Sigma_{\mathrm{int}}^\dagger{}&0&0\\0&\Sigma_{\mathrm{int}}&0\\0&0&\Sigma_{\mathrm{ext}}},P},
  \]
  and
  \begin{enumerate}[a)]
    \item
    $
    \bm{\Sigma_{\mathrm{int}}&0\\0&\Sigma{}_{\mathrm{ext}}}\bm{-A&-B\\C&D}\!=\!\bm{-A&-B\\C&D}^{\mathsf{T}}\!\bm{\Sigma_{\mathrm{int}}&0\\0&\Sigma{}_{\mathrm{ext}}},\,\bm{-A&-B\\C&D}\!+\!\bm{-A&-B\\C&D}^{\mathsf{T}}\!\!\succeq{}\!0;
    $
    \item $P\in\R^{n\times{}n}$ is a permutation matrix;
    \item $\Sigma_{\mathrm{int}}^\dagger{}\in\R^{n_{\mathrm{int}}^\dagger{}\times{}n_{\mathrm{int}}^\dagger{}},\Sigma_{\mathrm{int}}\in\R^{n_{\mathrm{int}}\times{}n_{\mathrm{int}}},\Sigma_{\mathrm{ext}}\in\R^{n_{\mathrm{int}}\times{}n_{\mathrm{int}}}$ are signature matrices where $n_{\mathrm{int}}^\dagger{}+n_{\mathrm{int}}=n_{\mathrm{L}}+n_{\mathrm{C}}$ and $\trace{\Sigma_{\mathrm{int}}^\dagger{}}+\trace{\Sigma_{\mathrm{int}}}=n_{\mathrm{C}}-n_{\mathrm{L}}$;
    \item
    $
    \bm{\Sigma_{\mathrm{int}}^\dagger{}&0\\0&\Sigma{}_{\mathrm{int}}}\bm{\Theta&\Gamma{}^{\mathsf{T}}\\\Gamma&\Phi}=\bm{\Theta&\Gamma{}^{\mathsf{T}}\\\Gamma&\Phi}^{\mathsf{T}}\bm{\Sigma_{\mathrm{int}}^\dagger{}&0\\0&\Sigma{}_{\mathrm{int}}},\;\bm{\Theta&\Gamma{}^{\mathsf{T}}\\\Gamma&\Phi}\succ{}0.
    $
  \end{enumerate}
\end{enumerate}
\end{theorem}

The proof of this theorem is given after a short discussion of \textit{(2)a)--d)}, and the definition of $\mathcal{F}\funof{\cdot}$. First observe that condition \textit{(2)a)} is precisely the condition \cref{eq:introst} from the introduction. The role of the permutation matrix in \textit{(2)b)} is simply to reorder the entries in the current and voltage vectors $i,v$ as appropriate. Condition \textit{(2)c)} constrains the sign patterns in the internal signature matrices $\Sigma_{\mathrm{int}}^\dagger{}$ and $\Sigma_{\mathrm{int}}$ to reflect the numbers of inductors and capacitors in the network. More specifically, there is one $-1$ in one of these matrices for each inductor, and one $+1$ for each capacitor. Finally \textit{(2)d)} defines a second condition that is similar in spirit to \cref{eq:introst}, though it involves the third signature matrix $\Sigma_{\mathrm{int}}^\dagger{}$ and the second block matrix. These matrices are in fact only needed to describe the internal currents and voltages, and will not be used in the next subsection. This condition can be viewed as an internal analogue of \textit{(2)a)}. Note that a third signature matrix is required since there may be a discrepancy between the state dimension in the state-space model and the number of inductors and capacitors in the RLCT network (the state dimension $n_{\mathrm{int}}$ may be less than $n_{\mathrm{C}}+n_{\mathrm{L}}$, see \cite{Hug17b} for an extensive discussion of why this can occur). 

The function $\mathcal{F}$ is given as follows:
\begin{equation}\label{eq:monster}
  \begin{aligned}
  \mathcal{F}\funof{\bm{-A&-B\\C&D},\bm{\Theta&\Gamma{}^{\mathsf{T}}\\\Gamma&\Phi},\Sigma,P}{}&=\\
  \tfrac{1}{2}
    \scalemath{0.8}{\sqfunof{
    \begin{array}{c;{2pt/2pt}c}
    P&0\\\hdashline[2pt/2pt]
    0&P
    \end{array}
    }}
    \scalemath{0.8}{\sqfunof{
    \begin{array}{c;{2pt/2pt}c}
    I+\Sigma{}&I-\Sigma{}\\\hdashline[2pt/2pt]
    I-\Sigma{}&I+\Sigma{}
    \end{array}
    }}&
    \scalemath{0.8}{\sqfunof{\begin{array}{ccc}
  -Q_1^{\mathsf{T}}\Phi^{-\frac{1}{2}}\Gamma{}^{\mathsf{T}}\Theta^{-1}&0&0\\
  Q_2^{\mathsf{T}}\Phi^{\frac{1}{2}}\Delta{}\Phi^{-\frac{1}{2}}A&Q_2^{\mathsf{T}}\Phi^{\frac{1}{2}}\Delta{}\Phi^{-\frac{1}{2}}B&0\\
  0&0&I\\\hdashline[2pt/2pt]
  -Q_1^{\mathsf{T}}\Gamma^{\mathsf{T}}\Phi^{-\frac{1}{2}}A&-Q_1^{\mathsf{T}}\Gamma^{\mathsf{T}}\Phi^{-\frac{1}{2}}B&0\\
  \Phi^{-\frac{1}{2}}\Delta\Phi^{-\frac{1}{2}}&0&0\\
  0&I&0
  \end{array}
  }}.
  \end{aligned}
\end{equation}
In the above $\Delta=\funof{\Phi-\Gamma{}\Theta^{-1}\Gamma{}^{\mathsf{T}}}^{\frac{1}{2}}$, and $Q_1,Q_2$ are matrices of (normalised) eigenvectors of $\Theta$ and $\Phi$ respectively, where the columns of $Q_1$ and $Q_2$ are ordered such that the matrices
\begin{equation}\label{eq:annoyingcom}
  \bm{Q_1&0\\0&Q_2}\;\text{and}\;\bm{\Sigma_{\mathrm{int}}^\dagger{}&0\\0&\Sigma_{\mathrm{int}}}
\end{equation}
commute. Note that while in general such an ordering of the eigenvectors may not exist, given matrices satisfying conditions \textit{(2)c)} and \textit{(2)d)} this will always be possible\footnote{\label{foo:1}To see this, observe that if
\[
\bm{I&0\\0&-I}\bm{X&Y^{\mathsf{T}}\\Y&Z}=\bm{X&Y^{\mathsf{T}}\\Y&Z}\bm{I&0\\0&-I},
\]
where $X$, $Y$ and $Z$ are compatibly dimensioned matrices, then $Y=0$. Therefore given matrices satisfying \textit{(2)c)} and \textit{(2)d)} from \Cref{thm:1}, there exist permutation matrices $P_1$ and $P_2$ such that $P_1\Theta{}P_1^{\mathsf{T}}$ and $P_2\Phi{}P_2^{\mathsf{T}}$ are block diagonal. We then see that these matrices admit eigendecompositions $P_1\Theta{}P_1^{\mathsf{T}}=\bar{Q}_1\bar{\Lambda}_1\bar{Q}_1^{\mathsf{T}}$ and $P_2\Phi{}P_2^{\mathsf{T}}=\bar{Q}_2\bar{\Lambda}_2\bar{Q}_2^{\mathsf{T}}$ such that $\bar{Q}_1$ and $\bar{Q}_2$ commute with $P_1\Sigma_{\mathrm{int}}^\dagger{}P_1^{\mathsf{T}}$ and $P_2\Sigma_{\mathrm{int}}P_2^{\mathsf{T}}$ respectively. It then follows that $Q_1=P_1^{\mathsf{T}}\bar{Q}_1P_1$ and $Q_2=P_2^{\mathsf{T}}\bar{Q}_2P_2$ are matrices of normalised eigenvectors of $\Theta$ and $\Phi$ such that the matrices in \cref{eq:annoyingcom} commute as required.}.

\begin{proof}
From the outset, it should be noted that in some sense this result is a specialisation of \cite[Theorem 5]{Hug17} to reciprocal networks, and the method of proof relies heavily on the methods developed there.

\textit{(1)}$\;\Rightarrow{}$\textit{(2)}: In \cite[Theorem 5]{Hug17} it is shown that if $\mathcal{B}$ is the behavior of a passive electrical network in which the signals $\funof{\vect{i},\vect{v}}\in\mathcal{B}$ correspond to the external and internal currents, and the external and internal voltages, respectively, then the underlying differential equations $G\funof{\tfrac{d}{dt}}$ and $H\funof{\tfrac{d}{dt}}$ that describe its behavior are highly structured. More specifically, given any such network, there exists a permutation matrix $P$ such that $\funof{\vect{i},\vect{v}}\in\mathcal{B}$ if and only if
\begin{equation}\label{eq:p1e1}
  \bm{
  P^{\mathsf{T}}\vect{i}\\
  P^{\mathsf{T}}\vect{v}
  }
  =\tfrac{1}{2}
  \bm{
  I+\Sigma{}&I-\Sigma{}\\
  I-\Sigma{}&I+\Sigma
  }
  \bm{
  \vect{e}\\
  \vect{h}
  }
,
\end{equation}
where $\Sigma$ is a signature matrix, and $\vect{e},\vect{h}\in\Lloc{,n}$ solve
\begin{equation}\label{eq:p1e2}
\begin{aligned}
\tfrac{d}{dt}\bm{\vect{e}_{\mathrm{int}}^\dagger{}\\\vect{h}_{\mathrm{int}}}&=
\bm{
\Lambda_1^{-1}&0\\0&\Lambda{}_2^{-1}
}
\bm{\vect{h}_{\mathrm{int}}^\dagger{}\\\vect{e}_{\mathrm{int}}},
\;\vect{e}=\bm{\vect{e}_{\mathrm{int}}^\dagger{}\\\vect{e}_{\mathrm{int}}\\\vect{e}_{\mathrm{ext}}},
\;\vect{h}=\bm{\vect{h}_{\mathrm{int}}^\dagger{}\\\vect{h}_{\mathrm{int}}\\\vect{h}_{\mathrm{ext}}},\\
&\bm{
\vect{e}_{\mathrm{int}}^\dagger{}\\-\vect{e}_{\mathrm{int}}\\\vect{e}_{\mathrm{ext}}
}
=\underbrace{
\bm{
0&-M_{21}^{\mathsf{T}}&0\\
M_{21}&M_{22}&M_{23}\\
0&M_{32}&M_{33}
}}_{\eqqcolon{}M}
\bm{
-\vect{h}_{\mathrm{int}}^\dagger{}\\\vect{h}_{\mathrm{int}}\\\vect{h}_{\mathrm{ext}}
},
\end{aligned}
\end{equation}
where $M\in\R^{n\times{}n}$, and $\Lambda_1\in\R^{n_{\mathrm{int}}^\dagger{}\times{}n_{\mathrm{int}}^\dagger{}},\Lambda_2\in\R^{n_{\mathrm{int}}\times{}n_{\mathrm{int}}}$. This is a lot of notation to take in, but each part of the representation in \cref{eq:p1e1,eq:p1e2} can be interpreted in terms of the underlying electrical network. In particular, \cref{eq:p1e1} is just defining a set of hybrid variables that correspond to the driving point currents and voltages (these are split between $\vect{e}_{\mathrm{ext}},\vect{h}_{\mathrm{ext}}$), and the internal currents and voltages (these are split between $\vect{e}_{\mathrm{int}},\vect{h}_{\mathrm{int}},\vect{e}_{\mathrm{int}}^\dagger{},\vect{h}_{\mathrm{int}}^\dagger{}$). The way the currents and voltages are divided between the hybrid variables is specified by the signature matrix $\Sigma$. The internal currents and voltages are split into two groups (with and without the $\dagger{}$s) to reveal the special internal structure in the matrix $M$. The entries in $M$ are determined by the resistances, the turns ratios in the transformers, and the network topology. By passivity and reciprocity, $M$ satsifies
\begin{equation}\label{eq:p1e14}
M+M^{\mathsf{T}}\succeq{}0,\;\Sigma{}M=M^{\mathsf{T}}\Sigma{},
\end{equation}
where $\Sigma$ is the aforementioned signature matrix (we will clarify the connection between $\Sigma$ and the internal and external signature matrices from the theorem statement below). The fact that certain entries of $M$ may be assumed to be zero for a certain ordering of the currents and voltages is not obvious, and is one of the key steps in the proof of \cite[Theorem 5]{Hug17}). Finally, the differential equation in \cref{eq:p1e2} is defining the element laws for the capacitors and inductors. This means that the matrices $\Lambda_1$ and $\Lambda_2$ are diagonal matrices with positive entries. 

Let us now begin to associate parts of \cref{eq:p1e1,eq:p1e2} with the matrices in \textit{(2)a)--d)}. We have already identified the permutation matrix \textit{(2)b)}. For \textit{(2)c)}, set
\[
\bm{-\Sigma_{\mathrm{int}}^\dagger{}&0&0\\
0&\Sigma_{\mathrm{int}}&0\\
0&0&\Sigma_{\mathrm{ext}}}=\Sigma.
\]
We observe from \cref{eq:p1e2} that each capacitor corresponds to exactly one positive entry in either $\Sigma_{\mathrm{int}}^\dagger{}$ or $\Sigma_{\mathrm{int}}$, and each inductor a negative entry (to see this observe from \cref{eq:p1e1} that if \emph{kk}-th element of $\Sigma$ is positive, the \emph{k}-th entry in \vect{e} is a current and the \emph{k}-th entry in \vect{h} is a voltage, and recall that the equation for a capacitor depends on the derivative of the voltage). For \textit{(2)a)}, set
\begin{equation}\label{eq:p1e20}
\bm{
  -A&-B\\
  C&D
}=
\bm{
  \Omega{}M_{22}\Omega{}&\Omega{}M_{23}\\
  M_{32}\Omega{}&M_{33}}
  ,\;\Omega{}=Q_2\funof{\Lambda{}_2+M_{21}\Lambda_1M_{21}^{\mathsf{T}}}^{-\frac{1}{2}}Q_2^{\mathsf{T}},
\end{equation}
where $Q_2$ is an orthogonal matrix that commutes with $\Sigma_{\mathrm{int}}$. This matrix inherits the required properties from \cref{eq:p1e14}. To see this, for positive semi-definiteness, first observe that
\[
\bm{-A&-B\\C&D}+\bm{-A&-B\\C&D}^{\mathsf{T}}=\bm{\Omega{}&0\\0&I}\funof{\bm{M_{22}&M_{23}\\M_{32}&M_{33}}+\bm{M_{22}&M_{23}\\M_{32}&M_{33}}^{\mathsf{T}}}
\bm{\Omega{}&0\\0&I}\succeq{}0,
\]
where the semi-definiteness follows from \cref{eq:p1e14}. For the symmetry property, note that from \cref{eq:p1e14} we get that $\Sigma_{\mathrm{int}}M_{21}=M_{21}\Sigma_{\mathrm{int}}^{\dagger{}}$, and so
\[
\Sigma_{\mathrm{int}}\funof{\Lambda_2+M_{21}\Lambda{}_1M_{21}^{\mathsf{T}}}=\Lambda_2\Sigma_{\mathrm{int}}+M_{21}\Sigma_{\mathrm{int}}^\dagger{}\Lambda_1M_{21}^{\mathsf{T}}=\funof{\Lambda_2+M_{21}\Lambda{}_1M_{21}^{\mathsf{T}}}\Sigma_{\mathrm{int}}.
\]
Therefore $\Sigma_{\mathrm{int}}$ commutes with $\Omega^{-2}$, and as a result with any $\Omega^{-2n}$ for $n\in\cfunof{0,1,2,\ldots{}}$. Now let $p\funof{x}$ be a polynomial such that for every eigenvalue $\lambda_k$ of $\Omega^{-2}$, $p\funof{\lambda_k}=1/\sqrt{\lambda{}_k}$. Such a polynomial is guaranteed to exist since $\Omega^{-2}\succ{}0$, and can be obtained using a Lagrange interpolating polynomial \cite[\S{9.011}]{JJ99}. We therefore have that $\Sigma_{\mathrm{int}}p\funof{\Omega^{-2}}=p\funof{\Omega^{-2}}\Sigma_{\mathrm{int}}$, and $p\funof{\Omega^{-2}}=\Omega$, meaning that $\Sigma_{\mathrm{int}}$ also commutes with $\Omega$. Therefore
\[
\bm{\Sigma_{\mathrm{int}}&0\\0&\Sigma_{\mathrm{ext}}}\bm{-A&-B\\C&D}=\bm{\Omega{}&0\\0&I}\bm{\Sigma_{\mathrm{int}}&0\\0&\Sigma_{\mathrm{ext}}}\bm{M_{22}&M_{23}\\M_{32}&M_{33}}\bm{\Omega{}&0\\0&I}=\bm{-A&-B\\C&D}^{\mathsf{T}}\bm{\Sigma_{\mathrm{int}}&0\\0&\Sigma_{\mathrm{ext}}}.
\]
For \textit{(2)d)}, set
\[
\bm{\Theta&\Gamma{}^{\mathsf{T}}\\\Gamma{}&\Phi}=\bm{Q_1\Lambda_1{}Q_1^{\mathsf{T}}&\Theta{}E^{\mathsf{T}}\Phi^{\frac{1}{2}}\\
\Phi^{\frac{1}{2}}E\Theta&Q_2\Lambda_2Q_2^{\mathsf{T}}},\;E=\Omega{}Q_2M_{21}Q_1^{\mathsf{T}},
\]
where $Q_1$ is an orthogonal matrix that commutes with $\Sigma_{\mathrm{int}}^\dagger{}$. The required symmetry property again follows from $\Sigma_{\mathrm{int}}M_{21}=M_{21}\Sigma_{\mathrm{int}}^{\dagger{}}$ and that $\Sigma_{\mathrm{int}}$ commutes with $\Omega$. Positive definiteness follows from $\Phi\succ{}0$ and that the Schur complement
\begin{equation}\label{eq:p1e16}
\Phi-\Gamma{}\Theta^{-1}\Gamma{}^{\mathsf{T}}=\Phi^{\frac{1}{2}}\funof{I-E\Theta{}E^{\mathsf{T}}}\Phi^{\frac{1}{2}}=\Phi^{\frac{1}{2}}\Omega{}\Phi{}\Omega\Phi^{\frac{1}{2}}\succ{}0,
\end{equation}
where the final equality in the above follows from
\[
\Omega^{-2}=\Phi+Q_2M_{21}Q_1^{\mathsf{T}}\Theta{}Q_1M_{21}^{\mathsf{T}}Q_2^{\mathsf{T}}\quad\Longleftrightarrow{}\quad{}I=\Omega\Phi{}\Omega{}+E\Theta{}E^{\mathsf{T}}.
\]
All that remains is to show that the behavior described by \cref{eq:p1e1,eq:p1e2} can be written in the state-space form in \textit{(2)}. To this end, let
\begin{equation}\label{eq:p1e5}
\begin{aligned}
\vect{u}=\vect{h}_{\mathrm{ext}},\;\Omega\vect{x}=\vect{h}_{\mathrm{int}},\;\vect{y}=\vect{e}_{\mathrm{ext}},\;\vect{t}=Q_1\vect{h}_{\mathrm{int}}^\dagger{},\\\vect{w}=\Omega{}\vect{e}_{\mathrm{int}},\;\vect{z}=Q_1\vect{e}_{\mathrm{int}}^\dagger{}\;\text{and}\;\Xi=I-E\Theta{}E^{\mathsf{T}}.
\end{aligned}
\end{equation}
Rewriting \cref{eq:p1e2} in terms of these parameters and variables and ordering the variables and equations appropriately gives
\begin{equation}\label{eq:p1e4}
\scalemath{0.8}{\sqfunof{\begin{array}{ccc;{2pt/2pt}ccc}
\Xi{}\tfrac{d}{dt}&0&0&0&-I&0\\
C&D&-I&0&0&0\\\hdashline[2pt/2pt]
0&0&0&I&0&-\Theta\tfrac{d}{dt}\\
-A&-B&0&-E&I&0\\
E^{\mathsf{T}}&0&0&0&0&I
\end{array}}}
\scalemath{0.8}{\sqfunof{\begin{array}{c}
\vect{x}\\
\vect{u}\\
\vect{y}\\\hdashline[2pt/2pt]
\vect{t}\\
\vect{w}\\
\vect{z}
\end{array}}}=0.
\end{equation}
Denoting the \emph{k}-th block of rows as $r_k$, performing the sequence of elementary row operations
\[
r_3\mapsto{}r_3+\Theta{}\tfrac{d}{dt}r_5,\;r_1\mapsto{}r_1+Er_3+r_4,\;r_3\mapsto{}r_3-\Theta{}E^{\mathsf{T}}r_1,\;r_4\mapsto{}r_4+Er_3
\]
shows that \cref{eq:p1e4} is equivalent to
\begin{equation}\label{eq:p1e17}
\scalemath{0.8}{\sqfunof{\begin{array}{ccc;{2pt/2pt}ccc}
\tfrac{d}{dt}I-A&-B&0&0&0&0\\
C&D&-I&0&0&0\\\hdashline[2pt/2pt]
\Theta{}E^{\mathsf{T}}A&\Theta{}E^{\mathsf{T}}B&0&I&0&0\\
-\Xi{}A&-\Xi{}B&0&0&I&0\\
E^{\mathsf{T}}&0&0&0&0&I
\end{array}}}
\scalemath{0.8}{\sqfunof{\begin{array}{c}
\vect{x}\\
\vect{u}\\
\vect{y}\\\hdashline[2pt/2pt]
\vect{t}\\
\vect{w}\\
\vect{z}
\end{array}}}=0.
\end{equation}
Through \cref{eq:p1e5}, this shows that $\funof{\vect{e},\vect{h}}$ solve \cref{eq:p1e2} if and only if
\[
\vect{e}=\bm{-Q_1^{\mathsf{T}}E^{\mathsf{T}}&0&0\\\Omega^{-1}\Xi{}A&\Omega^{-1}\Xi{}B&0\\0&0&I}\bm{\vect{x}\\\vect{u}\\\vect{y}},\;\vect{h}=\bm{-Q_1^{\mathsf{T}}\Theta{}E^{\mathsf{T}}A&-Q_1^{\mathsf{T}}\Theta{}E^{\mathsf{T}}B&0\\\Omega&0&0\\0&I&0}\bm{\vect{x}\\\vect{u}\\\vect{y}},\;\funof{\vect{x},\vect{u},\vect{y}}\in\beh.
\]
Finally we observe that $\Omega,E,Q_1$ and $Q_2$ can all be written in terms of $\Theta,\Gamma$ and $\Phi$. More specifically we see from \cref{eq:p1e16} that
\begin{equation}\label{eq:p1e18}
\Omega=\Phi^{-\frac{1}{2}}\funof{\Phi-\Gamma{}\Theta^{-1}\Gamma^{\mathsf{T}}}^{\frac{1}{2}}\Phi^{-\frac{1}{2}},
\end{equation}
that $E=\Phi^{-\frac{1}{2}}\Gamma{}\Theta^{-1}$, and that $Q_1,Q_2$ are the (normalised) matrices of eigenvectors of $\Theta$ and $\Phi$. We can therefore write the behavior as a function of the matrices in \textit{(2)a)--d)}, and the specific expression as given by the function $\mathcal{F}\funof{\cdot{}}$ in \cref{eq:monster} follows by substitution into \cref{eq:p1e1},

\textit{(2)} $\Rightarrow{}$\textit{(1)}: First observe that given matrices with the properties in \textit{(2)a)--d)} we can reverse the arguments from \cref{eq:p1e17} all the way back to \cref{eq:p1e2}. In particular, by \textit{(2)d)} $\Theta$ and $\Phi$ are symmetric and commute with $\Sigma_{\mathrm{int}}^\dagger{}$ and $\Sigma_{\mathrm{int}}$. Therefore they have eigendecompositions $\Theta=Q_1\Lambda{}Q_1^{\mathsf{T}}$ and $\Phi=Q_2\Lambda{}Q_2^{\mathsf{T}}$ such that $Q_1$ and $Q_2$ are orthogonal and commute with $\Sigma_{\mathrm{int}}^\dagger{}$ and $\Sigma_{\mathrm{int}}$, respectively (to see this second point, see \cref{foo:1}). We then get $E=\Phi^{-\frac{1}{2}}\Gamma{}\Theta^{-1}$ and $\Omega$ from \cref{eq:p1e18}, from which the entries of $M$ soon fall. $M$ inherits the properties in \cref{eq:p1e14} from \textit{(2)a)} and \textit{(2)d)}. 

Therefore it suffices to show that there exists an RLCT network with behavior given by the set of locally integrable solutions to \cref{eq:p1e1,eq:p1e2}, for any positive and diagonal $\Lambda{}_1$ and $\Lambda_2$, and $M$ that satisfies \cref{eq:p1e14}. The required construction is classical in the network synthesis literature, but we outline the relevant steps for completeness.  Following \cite{AV06}, \cref{eq:p1e1,eq:p1e2} can be synthesised by interconnecting a network of resistors and transformers, with a set of inductors and capacitors, with dynamics specified by the algebriac and differential equation in \cref{eq:p1e2} respectively. By choosing the inductances and capacitances appropriately, the differential equation can be synthesised for any diagonal matrices $\Lambda_1$ and $\Lambda_2$ with positive entries. Finally it is also shown in \cite{AV06} that any $M$ that satisfies \cref{eq:p1e14} can be synthesised using resistors and transformers. Therefore for any $\mathcal{B}$ given by \textit{(2)}, there exists an RLCT network with the same behavior as required.
\end{proof}

\subsection{External RLCT Network Behaviors and Minimal Realisations}

Electrical networks as defined in this paper interact with the world at large through the driving points. Therefore for the purposes of control system analysis and design it is often sufficient to only describe the set of external currents and voltages, which we shall refer to as the external network behavior. As we shall soon see, the external behavior of an RLCT network admits a simpler state-space description than that in \Cref{thm:1}. Furthermore for networks containing only subsets of the RLCT elements, things simplify further, and the resulting behaviors always admit observable and controllable state-space realisations.

The external behavior of an RLCT network is the projection of the full behavior (the set of solutions to \cref{eq:beh1}) onto the external currents and voltages:
\[
\cfunof{\funof{i_{\mathrm{ext}},v_{\mathrm{ext}}}\in\Lloc{,n_{\mathrm{ext}}}\times{}\Lloc{,n_{\mathrm{ext}}}:\funof{\bm{i_{\mathrm{int}}\\i_{\mathrm{ext}}},\bm{v_{\mathrm{int}}\\v_{\mathrm{ext}}}}\in\Lloc{,n}\times{}\Lloc{,n}\text{ satisfy \cref{eq:beh1}}}.
\]
A more compact description of the external behavior can be obtained by eliminating the internal currents and voltages $i_{\mathrm{int}}\in\Lloc{,n_{\mathrm{int}}}$ and $v_{\mathrm{int}}\in\Lloc{,n_{\mathrm{int}}}$ from \cref{eq:beh1}. There are enough equations in \cref{eq:beh1} to make this possible, but when eliminating functions in $\Lloc{}$ there are some extra technical considerations. The required technique is called proper elimination \cite{Pol93}. For brevity we omit the details, however, as shown in \cite{Hug17}, internal currents and voltages in RLCT networks (and more generally linear passive networks) are always properly eliminable (this can also be seen from \cref{eq:p1e17} in the proof of \Cref{thm:1}). After performing this step we obtain an equivalent description of the external behavior
\begin{equation}\label{eq:beh3}
\cfunof{\funof{i_{\mathrm{ext}},v_{\mathrm{ext}}}\in\Lloc{,n_{\mathrm{ext}}}\times{}\Lloc{,n_{\mathrm{ext}}}:G_{\mathrm{ext}}\funof{\tfrac{d}{dt}}i_{\mathrm{ext}}=H_{\mathrm{ext}}\funof{\tfrac{d}{dt}}v_{\mathrm{ext}}},
\end{equation}
where $G_{\mathrm{ext}}$ and $H_{\mathrm{ext}}$ are matrices of polynomials. This representation of the external behavior is in exactly the same form as the description of the full behavior in \cref{eq:beh11}, but now only the external currents and voltages remain.

It was shown in \cite{Hug19} that \cref{eq:beh3} gives the external behavior of an RLCT network if and only if \cref{eq:beh3} is equal to the external behavior of a state-space model with the structure in \cref{eq:introst} (though the state-space model is not necessarily observable or controllable, consider again the Bott-Duffin network in \Cref{fig:bottduffin}). The following theorem shows that the networks built out of subsets of the RLCT elements have state-space descriptions with the additional structure in \Cref{tab:1}. Note that as before these algebraic characterisations are equivalent. That is, every network behavior has a state-space realisation with the given algebraic structure, and also every structured state-space realisation can be synthesised with a network of the corresponding type.  

\begin{theorem}\label{thm:2}
Let
\[
\mathcal{B}=\cfunof{\funof{\vect{i}_{\mathrm{ext}},\vect{v}_{\mathrm{ext}}}\in\Lloc{,n_{\mathrm{ext}}}\times{}\Lloc{,n_{\mathrm{ext}}}:G_{\mathrm{ext}}\funof{\tfrac{d}{dt}}\vect{i}_{\mathrm{ext}}=H_{\mathrm{ext}}\funof{\tfrac{d}{dt}}\vect{v}_{\mathrm{ext}}},
\]
where $G_{\mathrm{ext}}$ and $H_{\mathrm{ext}}$ are matrices of polynomials. Given any entry $\bullet$ in \Cref{tab:1}, the following are equivalent:
\begin{enumerate}
  \item $\mathcal{B}$ is the external behavior of an RLCT network constructed using only elements from entry $\bullet$ in \Cref{tab:1}.
  \item
  \[
  \mathcal{B}=\cfunof{\funof{\vect{i}_{\mathrm{ext}},\vect{v}_{\mathrm{ext}}}:\funof{\vect{x},\vect{u},\vect{y}}\in{}\beh,
  \scalemath{0.8}{\sqfunof{
  \begin{array}{c}
  \vect{i}_{\mathrm{ext}}\\\hdashline[2pt/2pt]
  \vect{v}_{\mathrm{ext}}
  \end{array}
  }}
  =
  \scalemath{0.8}{\sqfunof{
  \begin{array}{cc;{2pt/2pt}cc}
  P_1&0&P_2&0\\\hdashline[2pt/2pt]
  0&P_2&0&P_1
  \end{array}
  }}
  \scalemath{0.8}{\sqfunof{
  \begin{array}{c}
  \vect{y}\\\hdashline[2pt/2pt]
  \vect{u}
  \end{array}
  }}},
  \]
  where $A,B,C$ and $D$ are as in entry $\bullet$ in \Cref{tab:1}, $\funof{C,A}$ and $\funof{A,B}$ are observable and controllable, and $\bm{P_1&P_2}\in\R^{n\times{}n}$ is a permutation matrix.
\end{enumerate}
\end{theorem}

\begin{proof}
Following the proof of \Cref{thm:1}, we see that $\funof{\vect{i}_{\mathrm{ext}},\vect{v}_{\mathrm{ext}}}\in\mathcal{B}$, where $\mathcal{B}$ is as in \textit{(1)}, if and only if
\begin{equation}\label{eq:p2e1}
\bm{P^{\mathsf{T}}\vect{i}_{\mathrm{ext}}\\P^{\mathsf{T}}\vect{v}_{\mathrm{ext}}}=\tfrac{1}{2}\bm{
  I+\Sigma_{\mathrm{ext}}&I-\Sigma_{\mathrm{ext}}\\I-\Sigma_{\mathrm{ext}}&I+\Sigma_{\mathrm{ext}}
}
\bm{
  \vect{e}_{\mathrm{ext}}\\\vect{h}_{\mathrm{ext}}
},
\end{equation}
where $\funof{\vect{e}_{\mathrm{ext}},\vect{h}_{\mathrm{ext}}}$ satisfy \cref{eq:p1e2} (and all notation is as in \Cref{thm:1}). It is no loss of generality to assume that
\[
\Sigma_{\mathrm{ext}}=\bm{I&0\\0&-I},
\]
and so \cref{eq:p2e1} can be rewritten as
\[
\scalemath{0.8}{\sqfunof{
  \begin{array}{c}
  \vect{i}_{\mathrm{ext}}\\\hdashline[2pt/2pt]
  \vect{v}_{\mathrm{ext}}
  \end{array}
  }}
  =
  \scalemath{0.8}{\sqfunof{
  \begin{array}{cc;{2pt/2pt}cc}
  P_1&0&P_2&0\\\hdashline[2pt/2pt]
  0&P_2&0&P_1
  \end{array}
  }}
  \scalemath{0.8}{\sqfunof{
  \begin{array}{c}
  \vect{e}_{\mathrm{ext}}\\\hdashline[2pt/2pt]
  \vect{h}_{\mathrm{ext}}
  \end{array}
  }},
\]
where $P=\bm{P_1&P_2}$ is a permutation matrix. Following the proof of \Cref{thm:1}, we see from \cref{eq:p1e5,eq:p1e17} that $\funof{\vect{e}_{\mathrm{ext}},\vect{h}_{\mathrm{ext}}}$ satisfy \cref{eq:p1e2} if and only if
\[
\bm{\vect{e}_{\mathrm{ext}}\\\vect{h}_{\mathrm{ext}}}=\bm{\vect{y}\\\vect{u}},\;\funof{\vect{x},\vect{u},\vect{y}}\in\beh,
\]
where $A,B,C,D$ satisfy the conditions in \textit{(2)} of \Cref{thm:1}. We then conclude that \textit{(1)} is equivalent to
\[
\mathcal{B}=\cfunof{\funof{\vect{i},\vect{v}}:\funof{\vect{x},\vect{u},\vect{y}}\in{}\beh,
\scalemath{0.8}{\sqfunof{
  \begin{array}{c}
  \vect{i}_{\mathrm{ext}}\\\hdashline[2pt/2pt]
  \vect{v}_{\mathrm{ext}}
  \end{array}
  }}
  =
  \scalemath{0.8}{\sqfunof{
  \begin{array}{cc;{2pt/2pt}cc}
  P_1&0&P_2&0\\\hdashline[2pt/2pt]
  0&P_2&0&P_1
  \end{array}
  }}
  \scalemath{0.8}{\sqfunof{
  \begin{array}{c}
  \vect{y}\\\hdashline[2pt/2pt]
  \vect{u}
  \end{array}
  }}}.
\]
To obtain the structured descriptions in \Cref{tab:1} note that for the signature matrices in \textit{(2)c)} of \Cref{thm:1}, it is no loss of generality to assume that
\[
\Sigma_{\mathrm{int}}^\dagger{}=\bm{I&0\\0&-I},\;\Sigma_{\mathrm{int}}=\bm{I&0\\0&-I},
\]
and furthermore if the network:
\begin{enumerate}[(i)]
  \item contains no inductors, then $\Sigma_{\mathrm{int}}^\dagger{}=I$, $\Sigma{}_{\mathrm{int}}=I$;
  \item contains no capacitors, then $\Sigma_{\mathrm{int}}^\dagger{}=-I$, $\Sigma{}_{\mathrm{int}}=-I$;
  \item contains no resistors, then
  \[
  \bm{-A&-B\\C&D}=-\bm{-A&-B\\C&D}^{\mathsf{T}}.
  \]
\end{enumerate}
 Note that we didn't explicitly show (iii) in \Cref{thm:1}, however by \cite[Theorem 3]{Hug17} if the network contains no resistors, then $M=-M^{\mathsf{T}}$, and so (iii) holds by considering \cref{eq:p1e20}. The structure in the entries of \Cref{tab:1} then just correspond to enforcing (i)--(iii) as appropriate. For example in the RCT case we see that (i) holds, and so \textit{(2)a)} in \Cref{thm:1} implies that if we partition $B,C$ and $D$ compatibly with $\Sigma_{\mathrm{ext}}$,
 \[
 \scalemath{.8}{\sqfunof{
 \begin{array}{c;{2pt/2pt}cc}
 A&B_1&B_2\\\hdashline[2pt/2pt]
 C_1&D_{11}&D_{12}\\
 C_2&D_{21}&D_{22}
 \end{array}}}
 =\scalemath{.8}{\sqfunof{
 \begin{array}{c;{2pt/2pt}cc}
 A&B_1&B_2\\\hdashline[2pt/2pt]
 -B_1^{\mathsf{T}}&D_{11}&D_{12}\\
 B_2^{\mathsf{T}}&-D_{12}^{\mathsf{T}}&D_{22}
 \end{array}}},\;\quad{}\text{where}\quad{}
 \bm{-A&B_1\\B_1^{\mathsf{T}}&D_{11}}\succeq{}0,\;D_{22}\succeq{}0.
 \]
The structure in the other entries in the table follow in a similar fashion. Therefore all that remains is to show that observable and controllable realisations may be used. As noted in \Cref{rem:cont} this can be deduced from existing results, but we give an alternative proof based on ideas that underpin the Kalman decomposition. From \cite[A.3]{Hug18} it follows that two behaviors with state-space realisations defined by
\[
\bm{A&B\\C&D}\quad{}\text{and}\quad{}\bm{A_1&B_1\\C_1&D_1}
\]
describe the same external behavior if and only if
\begin{enumerate}[(a)]
\item $C\funof{sI-A}^{-1}B+D=C_1\funof{sI-A_1}^{-1}B_1+D_1$;
\item There exist matrices $T,T_1$ such that $CA^kT=C_1A_1^k$ and $CA^k=C_1A_1^kT_1$ for $k=0,1,2,\ldots{}$
\end{enumerate}
We will now show that for each of the types of network it is possible to choose such an $A_1,B_1$ and $C_1$ with the same external behavior (and structural features in \Cref{tab:1}), where in addition $\funof{C_1,A_1}$ is observable and $\funof{A_1,B_1}$ is controllable. It in fact suffices to only consider RCT and LCT networks (RLT networks will be covered by duality with RCT networks, and all the other cases are subclasses of these networks). We start by following the standard approach for finding the controllable subspace of a realisation. Let 
\[
Q=\bm{Q_1&Q_2}
\]
be an orthogonal matrix such that $Q_1$ is a basis for the range space of the controllability matrix
\[
\bm{B&BA&\cdots{}&BA^{n-1}}.
\]
Then (see e.g. \cite[Theorem 3.6]{ZDG96}, noting from the proof that the similarity transformation can be assumed to be orthogonal)
\[
Q^{\mathsf{T}}B=\bm{\hat{B}\\0},\;Q^{\mathsf{T}}AQ=\bm{\hat{A}_{11}&\hat{A}_{12}\\0&\hat{A}_{22}}.
\]
Next note that for RCT networks $A=A^{\mathsf{T}}$, and for LCT networks $A=-A^{\mathsf{T}}$. Hence either
\[
\funof{Q^{\mathsf{T}}AQ}^{\mathsf{T}}=Q^{\mathsf{T}}AQ,\quad{}\text{or}\quad{}\funof{Q^{\mathsf{T}}AQ}^{\mathsf{T}}=-Q^{\mathsf{T}}AQ,
\]
and in either case this implies that $\hat{A}_{12}=0$. We also have in the RCT case that $C=-\Sigma{}_{\mathrm{ext}}B^{\mathsf{T}}$, and the LCT case that $C=B^{\mathsf{T}}$, and so
\[
CQ=-\Sigma{}_{\mathrm{ext}}\bm{\hat{B}^{\mathsf{T}}&0}\quad{}\text{or}\quad{}CQ=\bm{\hat{B}^{\mathsf{T}}&0}.
\]
Therefore we see that if
\[
\bm{A_1&B_1\\C_1&D_1}=\bm{Q_1^{\mathsf{T}}AQ_1&Q_1^{\mathsf{T}}B\\CQ_1&D},
\]
(a) and (b) hold (for (b), set $T=T_1^{\mathsf{T}}=Q_1$). Observe also from the above expressions that in the RCT case, $A_1,B_1$ and $C_1$ preserve the structural properties in \Cref{tab:1}. To see that the same is true in the LCT case (for an appropriate choice of the orthogonal basis) it is simplest to appeal to \cite[Part II, Theorem 8]{Wil72}.
\end{proof}

\begin{remark}\label{rem:cont}
In addition to the structure features in the matrices in \Cref{tab:1}, \Cref{thm:2} also shows that the state-space realisations of the external behaviors are observable and controllable. This will be significant in the next section, since it ensures that the special optimal control problems that we study always admit solutions. However it should be noted that the existence of observable and controllable realisations can be deduced from existing results, as we will now discuss. We have stated \Cref{thm:2} using the classical notions of observability and controllability of state-space models. However these concepts have been generalised to behavioral models (see \cite[Definitions 5.2.2 and 5.3.2]{PW98}). When the behavior in question is given by a state-space model, these notions are the same, which is why the classical definitions are adopted in this paper. In general RLCT networks need not have behaviorally controllable external behaviors, in which case they cannot be described by observable and controllable behavioral state-space models (see again the example in \Cref{fig:bottduffin}). However, as shown in \cite[Appendix D]{Hug17c}, whenever the external behavior of an RLCT network is behaviorally controllable, it admits an observable and controllable realisation. The existence of observable and controllable realisations for LCT networks can then be deduced from \cite[Remark 8]{Hug17c}, and for RLT and RCT networks from \cite[Theorem 9]{Hug18a} (this follows since this theorem demonstrates that if an RLCT network is not behaviorally controllable it must contain at least one inductor and one capacitor). \Cref{thm:2} offers an alternative proof of this fact based on standard state-space techniques.
\end{remark}

\section{Optimal Control of RLCT Networks}

In this section we switch our focus to $H_2$ and $H_{\infty}$ optimal control problems. Our objective is to highlight some special problems which, when the process to be controlled corresponds to one of the networks in \Cref{tab:1}, have analytical solutions. These analytical solutions have some curious properties from the perspective of the control of large-scale systems. In particular the control laws are either decentralised, or have naturally distributed implementations. Of particular interest with respect to structured synthesis is that the control laws are globally optimal (that is no performance is lost by using a decentralised or distributed control law), and that the notion of structure has nothing to do with sparsity. That is, the matrices that define the control law can be completely dense, yet by appealing to the electrical network structure, the control laws can be synthesised (and therefore implemented) in a highly structured fashion. The results we present are far from comprehensive, but hopefully illustrate the potential of this viewpoint for enhancing and extending structured synthesis methods. We illustrate the results using examples centred on large-scale optimisation, consensus, electrical power systems and heating networks.

We study the so called sub-optimal $H_2$ and $H_{\infty}$ control problems. In these problems we are given a process with inputs $w$ and $u$, and outputs $z$ and $y$. The controller measures the output $y$, and selects the inputs $u$. The objective is to design the control law so that either the $H_2$ or $H_\infty{}$ norm from $w$ to $z$ is less that some pre-specified level $\gamma>0$. This is summarised in the following problem:

\begin{problem}\label{prob:1}
Given $\gamma{}>0$ and a state-space model
\begin{equation}\label{eq:1}
\begin{aligned}
\tfrac{d}{dt}x&=Ax+B_1w+B_2u,\,x\funof{0}=0,\\
z&=C_1x+D_{11}w+D_{12}u,\\
y&=C_2x+D_{21}w,
\end{aligned}
\end{equation}
find a stabilising control law
\begin{equation}\label{eq:2}
\begin{aligned}
\tfrac{d}{dt}x_{\mathrm{K}}&=A_{\mathrm{K}}x_{\mathrm{K}}+B_{\mathrm{K}}y,\,x_{\mathrm{K}}\funof{0}=0,\\
u&=-C_{\mathrm{K}}x_{\mathrm{K}}-D_{\mathrm{K}}y,
\end{aligned}
\end{equation}
such that
\[
\norm{T_{zw}}_{\bullet}<\gamma{},
\]
where $\bullet$ denotes either the $H_2$ or $H_\infty{}$ norm, and $T_{zw}$ the transfer function from $w$ to $z$ defined by \cref{eq:1,eq:2}.
\end{problem}

Before focusing on our special problems, we first show that the signature symmetric structure in \cref{eq:introst} immediately leads to simplifications in the solution to \Cref{prob:1}. Of particular interest is that it is of no loss of generality to assume that the optimal control law inherits both the external and internal signature structure of the process. Since from \Cref{thm:1} we see that the internal signature specifies the types of reactive elements that are used in the synthesis of a system, this immediately suggests that extra insights into structured synthesis can be obtained from the electrical perspective. It can also be seen from the proof that the number of Riccati equations (or \glspl{lmi}) required to solve \Cref{prob:1} is halved, leading to a computational saving.

\begin{theorem}\label{thm:3}
Let the matrices in \cref{eq:1} of \Cref{prob:1} satisfy,
\begin{equation}\label{eq:sym1}
\scalemath{.8}{
\sqfunof{\begin{array}{c;{2pt/2pt}cc}
\Sigma_{\mathrm{int}}&0&0\\\hdashline[2pt/2pt]
0&\Sigma{}_{\mathrm{ext}}&0\\
0&0&\Sigma{}_{\mathrm{K}}
\end{array}
}}
\scalemath{.8}{
\sqfunof{\begin{array}{c;{2pt/2pt}cc}
-A&-B_1&-B_2\\\hdashline[2pt/2pt]
C_1&D_{11}&D_{12}\\
C_2&D_{21}&0
\end{array}
}}
=
\scalemath{.8}{
\sqfunof{\begin{array}{c;{2pt/2pt}cc}
-A&-B_1&-B_2\\\hdashline[2pt/2pt]
C_1&D_{11}&D_{12}\\
C_2&D_{21}&0
\end{array}
}}^{\mathsf{T}}
\scalemath{.8}{
\sqfunof{\begin{array}{c;{2pt/2pt}cc}
\Sigma_{\mathrm{int}}&0&0\\\hdashline[2pt/2pt]
0&\Sigma{}_{\mathrm{ext}}&0\\
0&0&\Sigma{}_{\mathrm{K}}
\end{array}
}},
\end{equation}
where $\Sigma_{\mathrm{int}},\Sigma_{\mathrm{ext}}$ and $\Sigma_{\mathrm{K}}$ are signature matrices. Then the following are equivalent:
\begin{enumerate}
  \item There exists a controller in the form of \cref{eq:2} that solves \Cref{prob:1}.
  \item There exists a controller in the form of \cref{eq:2}, where in addition
  \begin{equation}\label{eq:sym2}
  \scalemath{.8}{
  \sqfunof{\begin{array}{c;{2pt/2pt}c}
  \Sigma_{\mathrm{int}}&0\\\hdashline[2pt/2pt]
  0&\Sigma_{\mathrm{K}}
  \end{array}
  }}
  \scalemath{.8}{
  \sqfunof{\begin{array}{c;{2pt/2pt}c}
  -A_{\mathrm{K}}&-B_{\mathrm{K}}\\\hdashline[2pt/2pt]
  C_{\mathrm{K}}&D_{\mathrm{K}}
  \end{array}
  }}=
  \scalemath{.8}{
  \sqfunof{\begin{array}{c;{2pt/2pt}c}
  -A_{\mathrm{K}}&-B_{\mathrm{K}}\\\hdashline[2pt/2pt]
  C_{\mathrm{K}}&D_{\mathrm{K}}
  \end{array}
  }}^{\mathsf{T}}
  \scalemath{.8}{
  \sqfunof{\begin{array}{c;{2pt/2pt}c}
  \Sigma_{\mathrm{int}}&0\\\hdashline[2pt/2pt]
  0&\Sigma_{\mathrm{K}}
  \end{array}
  }},
  \end{equation}
  that solves \Cref{prob:1}.
\end{enumerate}
\end{theorem}
\begin{proof}
\textit{(2)} $\Rightarrow{}$\textit{(1)}: Immediate.

\textit{(1)} $\Rightarrow{}$\textit{(2)}: In both the $H_2$ and $H_{\infty}$ case, whenever there exists a solution to \Cref{prob:1}, a controller that satisfies \Cref{prob:1} can be obtained from the stabilising solutions of two Riccati equations (or more generally, from the feasible points of two LMIs). We will show the implication by demonstrating that if \cref{eq:sym1} holds, then the solution to the second Riccati equation (respectively feasible point to the second LMI) can be obtained directly from the first. This will be sufficient to guarantee that a controller obtained from these solutions has the structure in \cref{eq:sym2}, and meets the given performance requirement. For simplicity and to give the reader a single point of reference for both the $H_2$ and $H_{\infty{}}$ case, we will proceed under the `DKGF assumptions' (see i)--iv) in \cite[\S{III.A}]{DKGF89}\footnote{Note that \cite{DKGF89}ii) can be dropped, since we are working under the hypothesis that there exists a stabilising controller. Furthermore \cite{DKGF89}iv) can be dropped since by \cref{eq:sym1} it is equivalent to \cite{DKGF89}iii).}). The assumptions are thus:
\begin{enumerate}[{A}.1]
\item $\funof{A,B_1}$ is stabilisable and $\funof{C_1,A}$ is detectable;
\item $D_{12}^{\mathsf{T}}\bm{C_1&D_{12}}=\bm{0&I}$;
\item $D_{11}=0$.
\end{enumerate}
In each case we will point to the relevant literature to show that the claims continue to hold with A.1--3 removed.

The $H_2$ case: By \cite[Theorem 2]{DKGF89}, under A.1--3 the control law that minimises the $H_2$ norm of $T_{zw}$ can be obtained from the unique stabilising solutions (which are guaranteed to exist under A.1--3) to the Riccati equations
\begin{equation}\label{eq:rich2}
\begin{aligned}
A^{\mathsf{T}}X+XA-XB_2B_2^{\mathsf{T}}X+C_1^{\mathsf{T}}C_1&=0,\\
AY+YA^{\mathsf{T}}-YC_2^{\mathsf{T}}C_2Y+B_1B_1^{\mathsf{T}}&=0.
\end{aligned}
\end{equation}
Denote these stabilising solutions as $X=X_{H_2}$ and $Y=Y_{H_2}$. We now observe that under \cref{eq:sym1} $A=\Sigma_{\mathrm{int}}A^{\mathsf{T}}\Sigma{}_{\mathrm{int}}$, $\Sigma_{\mathrm{int}}B_1=-C_1^{\mathsf{T}}\Sigma_{\mathrm{ext}}$ and $\Sigma_{\mathrm{int}}B_2=-C_2^{\mathsf{T}}\Sigma_{\mathrm{K}}$, and so
\[
X_{H_2}=\Sigma_{\mathrm{int}}Y_{H_2}\Sigma_{\mathrm{int}}.
\] (
To see this simply pre and post multiply the first Riccati equation by $\Sigma_{\mathrm{int}}$, make the corresponding substitutions for $X_{H_2},A,B_2$ and $C_1$, and observe that you get the second Riccati equation. Therefore whenever a controller that satisfies \Cref{prob:1} exists, by \cite[Theorem 1]{DKGF89} the controller in the form of \cref{eq:2} with
\begin{equation}\label{eq:conth2}
\scalemath{.8}{
\sqfunof{\begin{array}{c;{2pt/2pt}c}
  -A_{\mathrm{K}}&-B_{\mathrm{K}}\\\hdashline[2pt/2pt]
  C_{\mathrm{K}}&D_{\mathrm{K}}
\end{array}
}}=
\scalemath{.8}{
\sqfunof{\begin{array}{c;{2pt/2pt}c}
-A+B_2B_2^{\mathsf{T}}X_{H_2}+\Sigma_{\mathrm{int}}X_{H_2}B_2B_2^{\mathsf{T}}\Sigma_{\mathrm{int}}&\Sigma_{\mathrm{int}}X_{H_2}B_2\Sigma_{\mathrm{K}}\\\hdashline[2pt/2pt]
B_2^{\mathsf{T}}X_{H_2}&0
\end{array}
}}
\end{equation}
satisfies \Cref{prob:1}. This controller clearly satisfies \cref{eq:sym2}. For the general case, to relax A.3, choose $D_{\mathrm{K}}$ such that $D_{11}=D_{12}D_{\mathrm{K}}D_{21}$ (if this can be done, it can be done consistently with \cref{eq:sym2}, and if it cannot then there exists no controller that solves \Cref{prob:1}). To relax A.1--2, see the treatment of the singular $H_2$ problem in \cite{Sto92}, and observe that in place of the two Riccati equations, we have a pair of LMIs (denoted in that paper by $F\funof{P}\succeq{}0,G\funof{Q}\succeq{}0$). An analogous argument then shows that $P$ is feasible for the first if and only if $Q=\Sigma_{\mathrm{int}}P\Sigma_{\mathrm{int}}$ is feasible for the second. These can be used to construct a controller meeting \cref{eq:sym2} through formulae analogous to the one above (a technical point here is that if a minimising sequence of controllers is required, perturbations to the process that preserve \cref{eq:sym1} will need to be used, c.f. the discussion at the end of \cite[\S{3}]{Sto92}, though this can always be done).

The $H_{\infty}$ case: By \cite[Theorem 3]{DKGF89}, under A.1--3 there exists a stabilising control law that solves \Cref{prob:1} if and only if there exist stabilising solutions $X=X_{H_\infty}$ and $Y=Y_{H_\infty}$ to the Riccati equations
\begin{equation}\label{eq:hinfric}
\begin{aligned}
A^{\mathsf{T}}X+XA-X\funof{B_2B_2^{\mathsf{T}}-\gamma^{-2}B_1B_1^{\mathsf{T}}}X+C_1^{\mathsf{T}}C_1&=0,\\
AY+YA^{\mathsf{T}}-Y\funof{C_2^{\mathsf{T}}C_2-\gamma^{-2}C_1^{\mathsf{T}}C_1}Y+B_1B_1^{\mathsf{T}}&=0,
\end{aligned}
\end{equation}
such that $X_{H_\infty}\succeq{}0$, $Y_{H_\infty}\succeq{}0$ and $\rho\funof{X_{H_\infty{}}Y_{H_\infty{}}}<\gamma{}^2$ (where $\rho\funof{\cdot}$ denotes the spectral radius of a matrix). The same observations as in the $H_2$ case show that $X_{H_\infty{}}=\Sigma_{\mathrm{int}}Y_{H_\infty{}}\Sigma_{\mathrm{int}}$. Furthermore a controller that solves \Cref{prob:1} is given by
\begin{equation}\label{eq:p3e2}
\scalemath{.8}{
\sqfunof{\begin{array}{c;{2pt/2pt}c}
  -A_{\mathrm{K}}&-B_{\mathrm{K}}\\\hdashline[2pt/2pt]
  C_{\mathrm{K}}&D_{\mathrm{K}}
\end{array}
}}=
\scalemath{.8}{
\sqfunof{\begin{array}{c;{2pt/2pt}c}
-A_\infty&Z_\infty{}\Sigma_{\mathrm{int}}X_{H_\infty}B_2\Sigma_{\mathrm{K}}\\\hdashline[2pt/2pt]
B_2^{\mathsf{T}}X_{H_\infty{}}&0
\end{array}
}},
\end{equation}
where
\[
\begin{aligned}
A_\infty{}&=A-\funof{B_2B_2^{\mathsf{T}}-\gamma{}^{-2}B_1B_1^{\mathsf{T}}}X_{H_\infty{}}-Z_\infty{}\Sigma{}_{\mathrm{int}}X_{H_\infty{}}B_2B_2^{\mathsf{T}}\Sigma_{\mathrm{int}}{},\\
Z_{\infty{}}&=\funof{I-\gamma^{-2}\Sigma_{\mathrm{int}}X_{H_\infty{}}\Sigma_{\mathrm{int}}X_{H_\infty{}}}^{-1}.
\end{aligned}
\]
Unlike in the $H_2$ case, this controller does not necessarily satisfy \cref{eq:sym2}. However, as alluded to in \cite[\S{}V.D]{DKGF89}, a coordinate transformation yields a more symmetric controller description, and this description does in fact satisfy \cref{eq:sym2} for processes satisfying \cref{eq:sym1}. To show this, begin by noting that even though $Z_{\infty}$ is not necessarily positive semi-definite, it has a square root that behaves very much like the unique positive semi-definite square root of a positive semi-definite matrix \cite[p.488]{HJ94}. To define this square root, consider the series expansion 
\[
\funof{1-x}^{-\frac{1}{2}}=1+x/2+3x^2/8+\ldots{},
\]
which converges for $\abs{x}<1$. Since $\rho\funof{\gamma^{-2}\Sigma_{\mathrm{int}}X_{H_\infty{}}\Sigma_{\mathrm{int}}X_{H_\infty{}}}<1$ this series will also converge if we set $x=\gamma^{-2}\Sigma_{\mathrm{int}}X_{H_\infty{}}\Sigma_{\mathrm{int}}X_{H_\infty{}}$, and so (by the Cayley Hamilton theorem) there exists a polynomial $p\funof{x}$ such that $Z_{\infty}=p\funof{\Sigma_{\mathrm{int}}X_{H_\infty{}}\Sigma_{\mathrm{int}}X_{H_\infty{}}}^2$. In a slight abuse of notation, we now write 
\[
Z_{\infty}^{\frac{1}{2}}=p\funof{\Sigma_{\mathrm{int}}X_{H_\infty{}}\Sigma_{\mathrm{int}}X_{H_\infty{}}}.
\]
Since $\Sigma_{\mathrm{int}}\funof{\Sigma_{\mathrm{int}}X_{H_\infty{}}\Sigma_{\mathrm{int}}X_{H_\infty{}}}=\funof{\Sigma_{\mathrm{int}}X_{H_\infty{}}\Sigma_{\mathrm{int}}X_{H_\infty{}}}^{\mathsf{T}}\Sigma_{\mathrm{int}}$, we then see that
\begin{equation}\label{eq:sr}
\begin{aligned}
\Sigma_{\mathrm{int}}Z_{\infty}^{\frac{1}{2}}&=\Sigma_{\mathrm{int}}p\funof{\Sigma_{\mathrm{int}}X_{H_\infty{}}\Sigma_{\mathrm{int}}X_{H_\infty{}}}=p\funof{\Sigma_{\mathrm{int}}X_{H_\infty{}}\Sigma_{\mathrm{int}}X_{H_\infty{}}}^{\mathsf{T}}\Sigma_{\mathrm{int}}\\
&=\funof{Z_{\infty}^{\frac{1}{2}}}^{\mathsf{T}}\Sigma_{\mathrm{int}}.
\end{aligned}
\end{equation}
Now perform the coordinate transformation
\[
\bm{A_{\mathrm{K}}&B_{\mathrm{K}}\\C_{\mathrm{K}}&D_{\mathrm{K}}}\mapsto{}\bm{{Z}^{-\frac{1}{2}}_\infty{}&0\\0&I}\bm{A_{\mathrm{K}}&B_{\mathrm{K}}\\C_{\mathrm{K}}&D_{\mathrm{K}}}\bm{{Z}^{\frac{1}{2}}_\infty{}&0\\0&I}.
\]
By inspection of \cref{eq:p3e2}, in these new coordinates \cref{eq:sym2} is satisfied if and only if
\[
\begin{aligned}
\Sigma_{\mathrm{int}}Z_\infty^{-\frac{1}{2}}A_\infty{}Z_\infty^{\frac{1}{2}}=\funof{Z_\infty^{-\frac{1}{2}}A_\infty{}Z_\infty^{\frac{1}{2}}}^{\mathsf{T}}\Sigma_{\mathrm{int}},\;\text{and}\\
\Sigma_{\mathrm{int}}Z_\infty^{\frac{1}{2}}\Sigma_{\mathrm{int}}X_{H_\infty}B_2\Sigma_{\mathrm{K}}=\funof{Z_\infty^{\frac{1}{2}}}^{\mathsf{T}}X_{H_\infty}B_2\Sigma_{\mathrm{K}}.
\end{aligned}
\]
The second of these equations holds by \cref{eq:sr}. For the first, observe (again using \cref{eq:sr}) that
\[
\Sigma_{\mathrm{int}}Z_{\infty}^{-\frac{1}{2}}A_\infty{}Z_{\infty}^{\frac{1}{2}}=\funof{Z_\infty^{\frac{1}{2}}}^{\mathsf{T}}\funof{\Sigma_{\mathrm{int}}Z_{\infty{}}^{-1}A_\infty{}\Sigma_{\mathrm{int}}}\funof{Z_\infty^{\frac{1}{2}}}^{\mathsf{T}}\Sigma_{\mathrm{int}},
\]
and so it is sufficient to show that $\Sigma_{\mathrm{int}}Z_{\infty{}}^{-1}A_\infty{}\Sigma_{\mathrm{int}}=\funof{Z_{\infty{}}^{-1}A_\infty{}}^{\mathsf{T}}$. Multiplying out the expression for $Z_{\infty{}}^{-1}A_{\infty}$ and using the identity
\[
X_{H_\infty{}}A-X_{H_\infty{}}\funof{B_2B_2^{\mathsf{T}}-\gamma{}^{-2}B_1B_1^{\mathsf{T}}}X_{H_\infty{}}=-A^{\mathsf{T}}X_{H_\infty{}}-\Sigma_{\mathrm{int}}B_1B_1^{\mathsf{T}}\Sigma_{\mathrm{int}}
\]
from \cref{eq:hinfric} shows that
\[
\begin{aligned}
Z_{\infty{}}^{-1}A_{\infty}=\gamma^{-2}&\Sigma_{\mathrm{int}}X_{H_\infty{}}\Sigma_{\mathrm{int}}A^{\mathsf{T}}X_{H_\infty{}}+A\\
&-\funof{B_2B_2^{\mathsf{T}}-\gamma{}^{-2}B_1B_1^{\mathsf{T}}}X_{H_\infty{}}-\Sigma_{\mathrm{int}}X_{H_\infty{}}\funof{B_2B_2^{\mathsf{T}}-\gamma{}^{-2}B_1B_1^{\mathsf{T}}}\Sigma_{\mathrm{int}}.
\end{aligned}
\]
The required property is readily checked from the above expression, and hence there exists a controller satisfying \cref{eq:sym2}. Just as in the $H_2$ case, when A.1--3 are relaxed the two Riccati equations are replaced with LMIs, and analogous arguments show that a controller satisfying \cref{eq:sym2} can be obtained. The relevant LMIs can be found in \cite{GA94} (for a suitable controller parametrisation see \cite{Gah92}).
\end{proof}

\subsection{LCT Networks}

We now turn our attention to our first special problem, which concerns LCT networks. LCT networks are extremely common in applications, and we will see connections to distributed optimisation, consensus and electrical power systems in the examples. The problem we consider is the special case of \Cref{prob:1} in which the disturbance $w$ corresponds to a process disturbance and sensor noise (the precise connection between \Cref{prob:1,prob:2} appears later as \cref{eq:p4e1}). The basic setup is illustrated in \Cref{fig:prob2}.

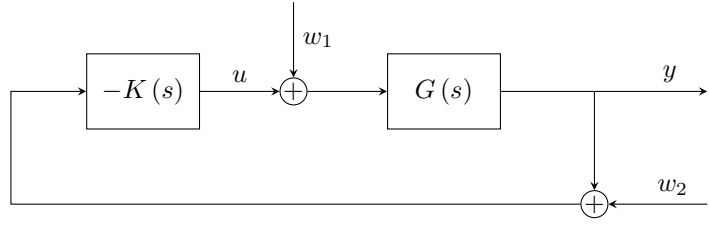
\begin{figure}
\centering
\begin{tikzpicture}[font=\normalsize,>=stealth]

  \node[rectangle, draw, minimum width=1.5cm, minimum height=1cm] (K) at
(0,0) {$-K\left(s\right)$};
  \node[rectangle, draw, minimum width=1.5cm, minimum height=1cm] (G) at
(4,0) {$G\left(s\right)$};
  \node[circle, draw, inner sep=0] (plus1) at (2,0) {$+$};
  \node[circle, draw, inner sep=0] (plus2) at (6,-1.5) {$+$};
  \coordinate (end) at (7.5,0);

  \draw[->] (plus2.west) -| ([xshift=-1cm]K.west) -- (K.west);
  \draw[->] ([yshift=-1.5cm]end) -- node[above, xshift=0.2cm] {$w_2$}
(plus2.east);
  \draw[->] (K.east) -- node[above] {$u$} (plus1.west);
  \draw[->] ([yshift=1cm]plus1.north) -- node[right] {$w_1$} (plus1.north);
  \draw[->] (plus1.east) -- (G.west);
  \draw[->] (G.east) -- ([xshift=-1cm]end) -- node[above] {$y$} (end);
  \draw[->] ([yshift=1.5cm]6,-1.5) -- (plus2.north);

\end{tikzpicture}
\caption{\label{fig:prob2}Block diagram representation of \Cref{prob:2}. This is a standard setup, in which the effects of process and sensor noise ($w_1$ and $w_2$) on the process output and control effort ($y$ and $u$) are to be minimised.}
\end{figure}

\begin{problem}\label{prob:2}
Given $\gamma{}>0$ and a state-space model
\begin{equation}\label{eq:ss1}
\begin{aligned}
\tfrac{d}{dt}x&=Ax+B\funof{u+w_1},\,x\funof{0}=0,\\
y&=Cx+w_2,\\
\end{aligned}
\end{equation}
find a stabilising control law
\begin{equation}\label{eq:ssc2}
\begin{aligned}
\tfrac{d}{dt}x_{\mathrm{K}}&=A_{\mathrm{K}}x_{\mathrm{K}}+B_{\mathrm{K}}y,\,x_{\mathrm{K}}\funof{0}=0,\\
u&=-C_{\mathrm{K}}x_{\mathrm{K}}-D_{\mathrm{K}}y,
\end{aligned}
\end{equation}
such that
\[
\norm{T_{zw}}_{\bullet}<\gamma{},
\]
where $\bullet$ denotes either the $H_2$ or $H_\infty{}$ norm, and $T_{zw}$ the transfer function from
\[
w=\bm{w_1\\w_2}\quad\text{to}\quad{}z=\bm{y\\u}
\]
defined by \cref{eq:ss1,eq:ssc2}.
\end{problem}

As shown in \Cref{tab:1}, LCT networks always admit a realisation where the $A,B$ and $C$ matrices satisfy\footnote{Note we are tacitly assuming the `$D$ matrix' in the process ($G\s$ in \Cref{fig:prob2}) is zero for simplicity, and to ensure that the $H_2$ case of \Cref{prob:2} is solvable for some $\gamma$ (though the statement of \Cref{thm:4} does not require this). Extensions to include direct terms are likely possible, but will not be pursued here.}
\begin{equation}\label{eq:exxxx2}
\mmfulla{A}{B}{C}{0}=
\scalemath{.8}{
    \sqfunof{\begin{array}{cc;{2pt/2pt}cc}
      0 & A_{12} & 0 & B_{12}\\
      -A_{12}^{\mathsf{T}} & 0 & B_{21} & 0\\\hdashline[2pt/2pt]
      0 & B_{21}^{\mathsf{T}} & 0 & 0 \\
      B_{12}^{\mathsf{T}} & 0 & 0 & 0
    \end{array}
    }}.
\end{equation}
The following theorem shows that when the process in \cref{eq:ss1} has dynamics of this form, \Cref{prob:2} has an analytical solution in both the $H_2$ and $H_\infty$ case. As explained in \Cref{ex:1,ex:2} below, both these control laws can be synthesised with an an electrical network that inherits the network structure of the process. This means that if the process is distributed (in the sense of the underlying topology of the LCT network), the controller is too, which coupled with the analytical nature of the solution, makes these control laws highly scalable. Interestingly in the $H_\infty{}$ case the control law is static and completely decentralised. As discussed in \Cref{rem:gof} it also has strong connections to robustness with respect to normalised coprime factors.

\begin{theorem}\label{thm:4}
If $A,B$ and $C$ are as in entry {\bf{LCT}} of \Cref{tab:1}, then the following are equivalent:
\begin{enumerate}
  \item There exists a controller in the form of \cref{eq:ssc2} that solves \Cref{prob:2}.
  \item The controller in the form of \cref{eq:ssc2} with
  \[
  \mmfulla{A_{\mathrm{K}}}{B_{\mathrm{K}}}{C_{\mathrm{K}}}{D_{\mathrm{K}}}=\begin{cases}
  \mmfull{A-2BB^{\mathsf{T}}}{B}{B^{\mathsf{T}}}{0}&\text{in the $H_2$ case;}\\
  \mmfull{}{}{}{\sqrt{2}I}&\text{in the $H_\infty{}$ case;}
  \end{cases}
  \]
  solves \Cref{prob:2}.
\end{enumerate}
\end{theorem}

Before proving this result, we will first make some comments and give some examples. The proof can be found at the end of the subsection.

\begin{remark}
The control laws in \Cref{thm:2} are optimal control laws (note that a stabilising control law always exists since by \Cref{thm:2} \funof{C,A} and \funof{A,B} are observable and controllable).
\end{remark}

\begin{remark}\label{rem:gof}
The $H_\infty{}$ case of \Cref{prob:2} is very closely related to the problem of minimising the $H_\infty$ norm of
\begin{equation}\label{eq:gof}
\bm{G\s\\I}\funof{I+K\s{}G\s{}}^{-1}\bm{K\s{}&I},
\end{equation}
where $G\s=C\funof{sI-A}^{-1}B$ and $K\s=C_{\mathrm{K}}\funof{sI-A_{\mathrm{K}}}^{-1}B_{\mathrm{K}}+D_{\mathrm{K}}$. This optimal control problem is special in the sense that all four closed-loop transfer functions are penalised, and is precisely the transfer function that should be minimised to obtain robust stability with respect to perturbations to the normalised coprime factors of $G\s$. It is also the central object in the loop shaping design procedure of McFarlane and Glover \cite{MG92}. It follows directly from the proof of \Cref{thm:3} that if $T_{zw}$ is replaced with \cref{eq:gof}, in the $H_\infty{}$ case \Cref{thm:3} holds with
\[
\mmfulla{A_{\mathrm{K}}}{B_{\mathrm{K}}}{C_{\mathrm{K}}}{D_{\mathrm{K}}}=\mmempty{I}.
\]
In the scalar case this can be understood graphically using the Riemann sphere as explained in \Cref{fig:riem}. 
\end{remark}

\begin{remark}
The equivalence in \Cref{thm:2} actually holds in the more general case that $A=-A^{\mathsf{T}}$ and $B=C^{\mathsf{T}}$. The behavior of such a state-space model need not correspond to that of a reciprocal network, which is why we haven't pursued this generalisation here (though there is in fact an equivalence between the behaviors of this class of state-space model, and the external behaviors of electrical networks constructed using only inductors, capacitors, transformers and gyrators).
\end{remark}

\begin{figure}

\begin{center}
\vspace{2cm}

\begin{tikzpicture}[>=stealth]

  \coordinate (origin) at (0,0);
  \coordinate (leftcircle) at (-2,0);
  \coordinate (rightcircle) at (2,0);
  \coordinate (topcircle) at (0,2);
  \coordinate (bottomcircle) at (0,-2);

  \draw (origin) circle (2cm);
  \draw[name path=ellipsy] (leftcircle) arc (180:360:2 and 0.6);
  \draw[densely dashed] (rightcircle) arc (0:180:2 and 0.6);

  \coordinate (projectedpoint) at ($(origin)+(225:2 and 0.6)$);
  \fill[fill=black] (projectedpoint) circle (0.05);

  \draw[densely dashed] (bottomcircle) -- (topcircle);
  \draw[->] (topcircle) -- +(0,1);

  \draw[name path=negativerealaxis] (bottomcircle) -- ($(bottomcircle)+(197:5)$);
  \draw[->] (bottomcircle) -- ($(bottomcircle)+(17:5)$) node[right] {$\mathrm{Re}$};

  \draw[->] (bottomcircle) -- ($(bottomcircle)+(173:5)$) node[left] {$\mathrm{Im}$};
  \draw[name path=negativeimaginaryaxis] (bottomcircle) -- ($(bottomcircle)+(-7:5)$);

  \draw[name path=northpole, densely dotted] (topcircle) -- ($(projectedpoint)!-2.9cm!(topcircle)$);

  \path [name intersections={of=northpole and negativerealaxis, by=minusone}];
  \fill[fill=black] (minusone) node[below] {-1} circle (0.05);

  \draw[thick, orange!80!black, postaction={on each segment={mid arrow=orange!80!black}}, densely dashed] (bottomcircle) arc (270:90:1.7 and 2);
  \draw[thick, orange!80!black, postaction={on each segment={mid arrow=orange!80!black}}] (topcircle) arc (90:-90:1.7 and 2);
  \draw[thick, orange!80!black, postaction={on each segment={mid arrow=orange!80!black}}] (bottomcircle) -- ($(bottomcircle)+(173:4)$);
  \draw[thick, orange!80!black, postaction={on each segment={mid arrow=orange!80!black}}] ($(bottomcircle)+(-7:4)$) -- (bottomcircle);

  \coordinate (projectedpointred) at (1.68,-0.32);
  \fill[fill=orange!80!black] (projectedpointred) circle (0.05);
  \draw[orange!80!black, name path=northpolered, densely dotted] (topcircle) -- ($(projectedpointred)!-2.6cm!(topcircle)$);
  \path [name intersections={of=northpolered and negativeimaginaryaxis, by=negativeintersect}];
  \fill[fill=orange!80!black] (negativeintersect) circle (0.05);

  \fill[fill=black] (topcircle) node[above, xshift=0.2cm] {N} circle (0.08 and 0.025);

\end{tikzpicture}

\end{center}
\caption[]{\label{fig:riem}Illustration of the observation about the optimal control law from \Cref{rem:gof}. In the scalar case, the $H_\infty$ norm of\newline{}
\begin{minipage}{\linewidth}
     $$
        \bm{G\s\\I}\funof{I+K\s{}G\s{}}^{-1}\bm{K\s{}&I},
     $$
  \end{minipage}
\vspace{.2cm}\newline{}
\noindent{}where $G\s=C\funof{sI-A}^{-1}B$ and $K\s=C_{\mathrm{K}}\funof{sI-A_{\mathrm{K}}}^{-1}B_{\mathrm{K}}+D_{\mathrm{K}}$ can be understood in terms of the Riemann sphere \cite[\S{2.4}]{Vin00}. In particular for any stabilising controller, the size of the $H_\infty{}$ norm of the above transfer function is given by one divided by the minimum chordal distance between the projections of the Nyquist plots of $G\jw$ and $1/K\jw$ onto the Riemann Sphere. Maximising robustness with respect to coprime factor perturbations therefore corresponds to maximising the chordal distances between the projections of $G\jw$ and $-1/K\jw$. The Nyquist plots of networks constructed out of inductors, capacitors and transformers (LCT networks) always lie on the imaginary axis. Therefore their projection onto the Riemann sphere corresponds to the great circle parallel to the imaginary axis, as shown in orange in the figure. The projection of the $-1$ point is at a chordal distance of $1/\sqrt{2}$ from the projection of the Nyquist plot for any LCT network for all frequencies, and in fact corresponds to the optimal control law. As shown in \Cref{thm:3} and \Cref{rem:gof}, the control law $K=I$ is optimal even in the multi-input-multi-output case.}

\vspace{2cm}
\end{figure}
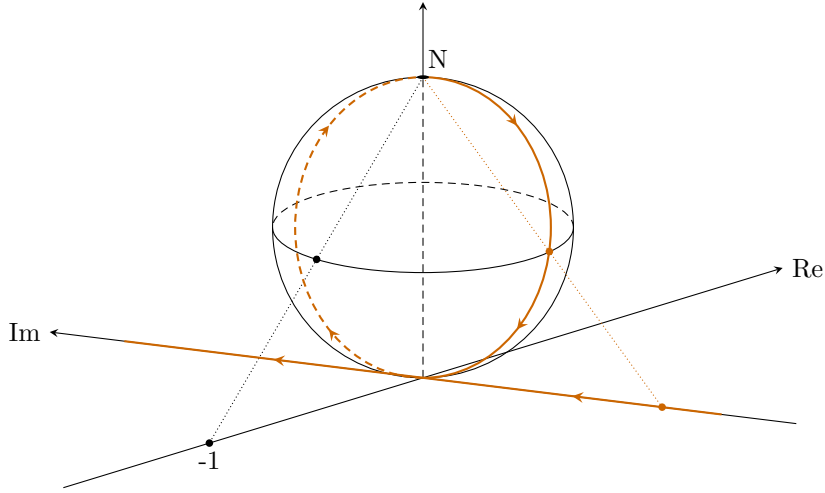
\begin{example}\label{ex:1}
The use of natural phenomena to solve optimisation problems has a rich history. In this example we illustrate this by constructing an electrical circuit, with natural connections to the $H_\infty{}$ case of \Cref{prob:2}, that solves the constrained least squares problem.

The objective of constrained least squares is to find an $\bar{x}\in\R^n$ that satisfies
\begin{equation}\label{eq:constls}
\begin{aligned}
\min_{\bar{x}\in\R^n}\norm{A\bar{x}-b}_2,
\text{s.t.}\,C\bar{x}=d,
\end{aligned}
\end{equation}
where $A\in\R^{m\times{}n}$, $C\in\R^{p\times{}n}$, $b\in\R^m$ and $d\in\R^p$ are the problem data. Constrained least squares encompasses a very broad class of problems including, for example, finite horizon LQR, and includes standard least squares and minimum norm solutions to a set of linear equations as special cases ($p=0$, and $A=I$ and $b=0$, respectively). The solution to \cref{eq:constls} can be obtained from the Karush-Kuhn-Tucker conditions
\begin{equation}\label{eq:kkt}
\bm{-A^{\mathsf{T}}A&-C^{\mathsf{T}}\\C&0}\bm{\bar{x}\\\bar{z}}=\bm{-A^{\mathsf{T}}b\\d},
\end{equation}
where we assume for simplicity that
\begin{enumerate}[i)]
\item $C$ is right invertible;
\item $\bm{A\\C}$ is left invertible;
\end{enumerate}
so that these equations have a unique solution. We will now see that the solution to \cref{eq:kkt} arises naturally from the $H_\infty{}$ solution to \Cref{prob:2}. 

To this end, consider the system
\begin{equation}\label{eq:lslct}
\begin{aligned}
\tfrac{d}{dt}x&=\bm{0&-C^{\mathsf{T}}\\C&0}x+\bm{A^{\mathsf{T}}\\0}\funof{u+w_1-r_1}+\bm{0\\I}r_2,\\
y&=\bm{A&0}x+w_2.
\end{aligned}
\end{equation}
The system matrices take the form of those in the {\bf{LCT}} entry of \Cref{tab:1}, and therefore \cref{eq:lslct} can be synthesised using $p$ inductors, $n$ capacitors and some transformers (note it is not immediate from the table that $r_2$ can be incorporated, but in fact it corresponds to inserting a driving point voltage in series with each of the inductors). When the reference signals $r_1$, $r_2$ are zero, this system is in precisely the form of \cref{eq:ss1} in \Cref{prob:2}, and so \Cref{thm:4} applies. In light of \Cref{rem:gof} the control law
\begin{equation}\label{eq:ctr}
u=-y,
\end{equation}
which can be synthesised by connecting unit resistors across each of the external terminal pairs, maximises robustness with respect to normalised coprime factor perturbations. We also see from \cref{eq:lslct} that the closed loop system becomes
\[
\begin{aligned}
\tfrac{d}{dt}x&=\bm{-A^{\mathsf{T}}A&-C^{\mathsf{T}}\\C&0}x+\bm{A^{\mathsf{T}}\\0}\funof{w_1-r_1}+\bm{0\\I}r_2,\\
y&=\bm{A&0}x+w_2.
\end{aligned}
\]
Therefore by applying the step inputs $r_1=bH\tm$ and $r_2=dH\tm$, where $H\tm$ denotes the unit step, we see that
\[
\lim_{t\rightarrow{}\infty}x\tm=\bm{\bar{x}\\\bar{z}}.
\]
This means that for any $b$ and $d$, the solution to the constrained least squares problem can be obtained by measuring the voltage across the capacitors. Of course it is not necessary to actually use an electrical circuit to implement \cref{eq:lslct,eq:ctr}. These equations could equally well be interpreted as a simple algorithm which can be easily distributed in the case of sparse $A$ and $C$, and \Cref{thm:4} then shows that this algorithm has excellent robustness properties.

It is also interesting to note that the standard consensus algorithm \cite{RBA05} also arises as a special case of the above, by setting $p=0$ and $A$ to be an incidence matrix. The corresponding circuit can in fact be realised without transformers, and \Cref{thm:4} again demonstrates that this simple algorithm has excellent robustness properties\footnote{A small technical point here is that (ii) no longer holds, and the realisation in \cref{eq:lslct} is neither controllable nor observable. However the external behavior of the circuit is guaranteed to be observable and controllable by \Cref{thm:2}, and so the control laws from \Cref{thm:4} are still optimal.}.
\end{example}

\begin{figure}
\centering
\ctikzset{bipoles/length=1cm}
\ctikzset{bipoles/thickness=1}

\subfloat[Electrical analogue of swing equation model.]{
\begin{tikzpicture}[scale=.6,font=\normalsize,>=stealth]

  \coordinate (origin) at (0,0);
  \coordinate (a1) at (2.5,0.5);
  \coordinate (a2) at (-2.5,0.5);
  \coordinate (a3) at (-0.5,-2.5);
  \coordinate (a3belowshifted) at (-1.5,-4);
  \coordinate (a3below) at (-0.5,-4);
  \coordinate (a2below) at (-2.5,-1.5);
  \coordinate (a1below) at (2.5,-1.5);

  \node[above] at (a1) {$1$};
  \node[above] at (a2) {$2$};
  \node[above, left] at (a3) {$3$};

  \fill[black] (origin) circle (0.1);
  \fill[black] (a1) circle (0.1);
  \fill[black] (a2) circle (0.1);
  \fill[black] (a3) circle (0.1);

  \draw (origin) to [L] (a1);
  \draw (a2) to [L] (origin);

  \draw[<-] ([xshift=0.15cm]a1) -- +(1,0) node[right, xshift=-0.1cm,
yshift=0.1cm] {$p_{\mathrm{G},1}$};
  \draw[->] ([yshift=0.2cm]a1below) -- node[right] {$\omega_1$}
([yshift=-0.3cm]a1);
  \draw (a1below) node[ground]{} to +(-1,0) to +(1,0);

  \draw[<-] ([xshift=-0.15cm]a2) -- +(-1,0) node[left, xshift=0.15cm,
yshift=0.1cm] {$p_{\mathrm{G},2}$};
  \draw[->] ([yshift=0.2cm]a2below) -- node[left] {$\omega_2$}
([yshift=-0.3cm]a2);
  \draw (a2below) node[ground]{} to +(-1,0) to +(1,0);

  \draw (a3) to [L] (origin);
  \draw (a3below) node[ground]{} to +(-1,0) to +(1,0);
  \draw (a3) to [C] (a3below);
  \draw[->] ([yshift=0.15cm]a3belowshifted) to node[left] {$\omega_3$}
+(0,1.4);

\end{tikzpicture}
}

\vspace{.5cm}
\subfloat[Electrical analogue of the optimal controller.]{
\begin{tikzpicture}[font=\normalsize,scale=.6,>=stealth]

  \coordinate (origin) at (0,0);
  \coordinate (a1) at (2.5,0.5);
  \coordinate (a2) at (-2.5,0.5);
  \coordinate (a3) at (-0.5,-2.5);
  \coordinate (a3belowshifted) at (-1.5,-4);
  \coordinate (a3below) at (-0.5,-4);
  \coordinate (a2below) at (-2.5,-1.5);
  \coordinate (a1below) at (2.5,-1.5);

  \node[above] at (a1) {$1$};
  \node[above] at (a2) {$2$};
  \node[above, left] at (a3) {$3$};

  \fill[black] (origin) circle (0.1);
  \fill[black] (a1) circle (0.1);
  \fill[black] (a2) circle (0.1);
  \fill[black] (a3) circle (0.1);

  \draw (origin) to [L] (a1);
  \draw (a2) to [L] (origin);

  \draw[->] ([yshift=-0.25cm]a1) -- ([yshift=-0.15cm]a1);
  \draw ([yshift=-0.15cm]a1) to node[yshift=-0.05cm,right] {$u_1$}
+(0,-1.45);

  \fill (a1) circle (0.1);
  \fill (a1below) circle (0.1);
  \draw ([yshift=0.2cm]a1below) -- ([yshift=-0.3cm]a1);

  \draw ([xshift=2cm]a1below) node[right, xshift=.1cm,yshift=0.625cm] {$2\,\Omega$}
to [R] ([xshift=2cm]a1);
  \draw ([xshift=4cm]a1below) node[right, yshift=0.6cm, xshift=0.25cm]
{$y_1$} to [I] ([xshift=4cm]a1);
  \draw ([xshift=4cm]a1below) node[ground]{} -- (a1below);
  \draw ([xshift=4cm]a1) -- (a1);

  \draw[->] ([yshift=-0.25cm]a2) -- ([yshift=-0.15cm]a2);
  \draw ([yshift=-0.15cm]a2) to node[yshift=-0.05cm,right] {$u_2$}
+(0,-1.45);

  \fill (a2) circle (0.1);
  \fill (a2below) circle (0.1);
  \draw ([yshift=0.2cm]a2below) -- ([yshift=-0.3cm]a2);

  \draw ([xshift=-2cm]a2below) node[left, xshift=-.1cm,yshift=0.625cm] {$2\,\Omega$}
to [R] ([xshift=-2cm]a2);
  \draw ([xshift=-4cm]a2below) node[left, yshift=0.6cm, xshift=-0.25cm]
{$y_2$} to [I] ([xshift=-4cm]a2);
  \draw ([xshift=-4cm]a2below) node[ground]{} -- (a2below);
  \draw ([xshift=-4cm]a2) -- (a2);

  \draw (a3) to [L] (origin);
  \draw (a3) to [C] (a3below) to +(1,0) to +(1,-0.2) node[ground]{};
  \draw (a3below) to (a3belowshifted);
+(0,1.4);

\end{tikzpicture}
}
 \caption{\label{fig:ex2} Optimal controller synthesis for a simple power system model. {\footnotesize \textsc{({\tiny A})}} illustrates the construction of the electrical analogue of a swing equation model for a small electrical power system. The same process applies to systems of any size. In this analogue, currents correspond to power flows, and nodal voltages to electrical frequencies (the rate of change of electrical angle). The generator buses (the labeled nodes) give the locations where either renewable generators or conventional synchronous machines are connected. The inductors that connect the nodes correspond to the transmission lines. The renewable generators (nodes 1--2) inject currents into the network corresponding to the deviations in power production from the renewable source. Conventional generators are modelled as capacitors, where the capacitance reflects the inertia in the machine. {\footnotesize \textsc{({\tiny B})}} illustrates an electrical realisation of the optimal control law for this network, as derived in \Cref{ex:2}. The current generated by the \emph{k}-th current source in the controller is set to equal the power imbalance along the transmission lines connected at the \emph{k}-th renewable generator bus (i.e. the power that must be injected by the generator to enforce conservation of power with the given set of flows along the lines). Each renewable generator in the network then sets its electrical frequency to equal the corresponding voltage $u_k$, as determined by this electrical network. The control law inherits the structure of the electrical network (it has the same topology, except for the addition of a set of $2\,\Omega$ resistors), giving it a simple and highly scalable distributed implementation.}
\end{figure}
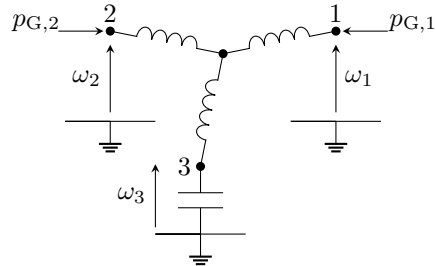
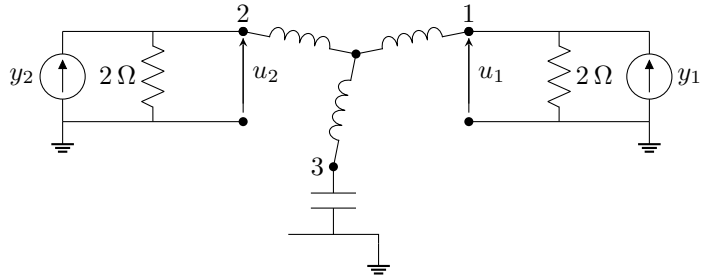

\begin{example}\label{ex:2}
Electrical power systems are, by design, lightly damped electro-mechanical systems. As such, their dynamics can be well modelled by {\bf{LCT}} networks. \Cref{thm:4} therefore applies, and we will show below that the $H_2$ optimal controller from \Cref{thm:4} has a natural distributed implementation in this context. This controller could be used as a starting point to design systems to combat stochastic disturbances arising from, for example, the increased use of renewable energy sources in electrical grids.

Mechanical and equivalent circuit analogues have long been a mainstay of power system modelling and design. Indeed many power system analysis tools, including large-scale methods for analysing inter-area oscillations (see for example \cite{BBS82}, or \cite[Chapter 12]{Kun94} for an introduction) are based on the validity of modelling power systems with lightly damped mechanical networks. This connection can be clearly seen in the swing equation power system model, which itself has an elegant mechanical analogue \cite{DCB13}, though in fact the lossless approximation remains valid even in high fidelity power system models (see for example the models in \cite{Kun94}, especially the equivalent circuit descriptions of synchronous machines in Chapter 5 and the transmission line descriptions in Chapter 6). In the swing equation power system model, the power system is modelled as a set of $n$ generator buses, interconnected through the transmission network. For simplicity we suppose that the generator buses are either renewable generators (the first $q$ buses) or synchronous machines, as modelled by
\begin{equation}\label{eq:gens}
\begin{cases}
p_{\mathrm{G},k}-p_{\mathrm{N},k}=0&\text{if $k\in\cfunof{1,\ldots,q}$,}\\
M_k\tfrac{d}{dt}\omega_k=-p_{\mathrm{N},k}&\text{otherwise.}
\end{cases}
\end{equation}
In the above $M_k>0$ is an inertia parameter reflecting the inertia of the given synchronous machine. Note that it is common to also include damping terms in these equations. However the damping is typically very small (indeed standard power system design requirements only call for damping ratios of 3--5\%), and so we assume these to be zero. The variable $\omega_k$ corresponds to the electrical frequency (the rate of change of electrical angle) at the \emph{k}-th bus, and $p_{\mathrm{G},k}$ $p_{\mathrm{N},k}$ denote the deviations in the power generated at the \emph{k}-th bus, and the power injected into the \emph{k}-th bus through the transmission network, respectively.

The dynamics of the generation buses are coupled through the dynamics of the transmission network. After linearisation, the transmission network dynamics are given by
\begin{equation}\label{eq:trans}
\begin{aligned}
\tfrac{d}{dt}\bm{p_{\mathrm{N},1}\\\vdots{}\\p_{\mathrm{N},q}}&=Y\bm{\omega_1\\\vdots{}\\\omega_q},
\end{aligned}
\end{equation}
where $Y$ is a Kron reduced weighted Laplacian matrix. That is
\[
Y=L_{11}-L_{12}L_{22}^{-1}L_{21},
\]
where
\[
L=\bm{L_{11}&L_{12}\\L_{21}&L_{22}}
\]
is a weighted Laplacian matrix. $L$ is sparse, with a sparsity pattern that reflects the topology of the grid (see for example \cite[Chapter 6]{Kun94}). 

There is a very natural electrical analogue associated with \cref{eq:gens,eq:trans}, as illustrated in \Cref{fig:ex2}. In this setting, currents correspond to power flows, and nodal voltages electrical frequencies. Each of the renewable generator equations in \cref{eq:gens} is associated with a driving point with driving point current $p_{\mathrm{G},k}$, and each synchronous machine equation with a grounded capacitor, with capacitance $M_k$. The transmission network is analogous to a network of inductors. The topology of this network mirrors that of the actual electrical power system, and the values of the inductances can be obtained from the weights in the weighted Laplacian matrix $L$. Therefore by \Cref{thm:2} this network is guaranteed to admit a controllable and observable state-space realisation with the structure in the {\bf{LCT}} entry of \Cref{tab:1}. In fact it is not hard to additionally show that the model takes the form
\[
\begin{aligned}
\tfrac{d}{dt}x&=\bm{0&A_{12}\\-A_{12}^{\mathsf{T}}&0}x+\bm{B_1\\0}\bm{\omega_1\\\vdots{}\\\omega_q},
\\\bm{p_{\mathrm{G},1}\\\vdots{}\\p_{\mathrm{G},q}}&=\bm{B_1^{\mathsf{T}}&0}x.
\end{aligned}
\]
This suggests that if it is possible to emulate any desired power-frequency relation between $p_{\mathrm{G}}$ and $\omega$ using our renewable generation sources, in order to minimise the effects of stochastic fluctuations in the power generated from the renewable sources (this is akin to the disturbance $w_2$ that acts on the output in \Cref{prob:2}), we should design this relationship according to $H_2$ control law in \Cref{thm:4}. An interesting point here, and the one that we want to emphasise with this example, is that although the matrices in the state-space realisation of the resulting control law are typically dense (this is generally the case since the matrix $Y$ is often dense due to $L_{22}^{-1}$ term), the control law can be implemented\footnote{At least from the perspective of calculating the relevant signals. There are of course likely a range of pratical challenges associated with implementing such a control law. However modern power electronics, in particular grid forming inverters \cite{PPG07}, offer tremendous flexibility in the control of the power and electrical frequency at the terminals of a renewable generation source. Nevertheless, not least because the power system model considered here is highly simplified, such a control law should only ever be used to provide insights, perhaps in terms of achievable performance or sensible control architectures, as part of more detailed design process.} in a distributed and highly scalable manner. To see this, observe that the dynamics of the control law from \Cref{thm:3} can be obtained from \cref{eq:exxxx2} (the dynamics of the process being controlled) by implementing the feedback $u=-2y$. This means that we can synthesise a system that will replicate the dynamics of the control law by building a copy of the electrical analogue of the power system, and placing a two ohm resistor over each pair of terminals where a renewable source is connected. This is in effect a distributed implementation of the optimal control law which inherits the sparsity structure of the electrical grid directly, and can be updated in a systematic and transparent manner as power system components are added and removed. This is illustrated in \Cref{fig:ex2}.
\end{example}

We now give the proof of \Cref{thm:3}.

\begin{proof}
\textit{(2)} $\Rightarrow{}$\textit{(1)}: Immediate.

\textit{(1)} $\Rightarrow{}$\textit{(2)}: \Cref{prob:2} is the following special case of \Cref{prob:1}:
\begin{equation}\label{eq:p4e1}
\begin{aligned}
\tfrac{d}{dt}x&=Ax+\bm{B&0}w+Bu,\,x\funof{0}=0,\\
z&=\bm{C\\0}x+\bm{0\\I}u,\\
y&=Cx+\bm{0&I}w.
\end{aligned}
\end{equation}
Under the hypothesis of \textit{(1)}, \funof{C,A} is detectable and \funof{A,B} is stabilisable. Therefore A.1--3 from the proof of \Cref{thm:3} hold, and a controller that solves \Cref{prob:2} can be obtained from the stabilising solution to a single Riccati equation in the $H_2$ case. A little more work is required in the $H_\infty$ case to obtain an optimal controller. This is required since \textit{(2)} promises a single controller for all $\gamma$ such that a solution exists, whereas the central controller in \cref{eq:p3e2} depends on $\gamma$, but again the Riccati equation provides all necessary insight.

$H_2$ case: Since $A,B$ and $C$ are as in entry {\bf{LCT}} of \Cref{tab:1}, $A=-A^{\mathsf{T}}$ and $B=C^{\mathsf{T}}$. By comparison with \cref{eq:rich2}, we see that the Riccati equation that must be solved in this case is
\[
XA+A^{\mathsf{T}}X-XBB^{\mathsf{T}}X+BB^{\mathsf{T}}=0.
\]
By inspection this equation has stabilising solution $X=I$, from which the expression for the controller follows by \cref{eq:conth2} (we can be sure that this is the stabilising solution since it is positive semi-definite, c.f. \cite[Corollary 13.8]{ZDG96}).

$H_\infty{}$ case: By comparison with \cref{eq:hinfric}, for this case we see that we require the stabilising solution to
\[
XA+A^{\mathsf{T}}X-\funof{1-\gamma^{-2}}XBB^{\mathsf{T}}X+BB^{\mathsf{T}}=0,
\]
and that the solution be positive semi-definite. This time we get the solution $X=X_{H_\infty}$ where
\begin{equation}\label{eq:p4e11}
X_{H_\infty}=\frac{\gamma}{\sqrt{\gamma^2-1}}I,\,\gamma>1.
\end{equation}
Again we can be certain that the above is the stabilising solution for $\gamma>1$ by \cite[Corollary 13.8]{ZDG96}, and note in particular that this implies that
\[
A-\sqrt{2}BB^{\mathsf{T}}
\]
is stable. Unlike the $H_2$ case, we must additionally satisfy the constraint
\[
\rho\funof{X_{H_\infty}\Sigma_{\mathrm{int}}X_{H_\infty}\Sigma_{\mathrm{int}}}<\gamma^2.
\]
When combined with \cref{eq:p4e11} this implies that the sub-optimal $H_{\infty}$ control problem is solvable only if $\gamma>\sqrt{2}$. Now consider the control law $u=-\sqrt{2}y$ from the theorem statement. We see from \cref{eq:p4e1} that the resulting closed loop system has realisation
\[
\mmfulla{A}{B}{C}{D}
=\scalemath{.8}{\sqfunof{
  \begin{array}{c;{2pt/2pt}cc}
  A-\sqrt{2}BB^{\mathsf{T}}&B&-\sqrt{2}B\\\hdashline[2pt/2pt]
  B^{\mathsf{T}}&0&0\\-\sqrt{2}B^{\mathsf{T}}&0&-\sqrt{2}I
  \end{array}
}}.
\]
It then follows from the bounded real lemma that this system has $H_\infty{}$ norm less than $\gamma$ if and only if there exists an $X\succ{}0$ such that
\[
XA-\tfrac{1}{\gamma\funof{\gamma^2-2}}\funof{\tfrac{1}{2}\funof{3\gamma^2-2}XBB^{\mathsf{T}}X+\sqrt{2}\gamma{}^3XBB^{\mathsf{T}}+\tfrac{1}{2}\funof{3\gamma^2-2}I}+\funof{\star}^{\mathsf{T}}\prec{}0.
\]
Setting $X=I$ verifies this LMI for all $\gamma>\sqrt{2}$, which completes the proof.
\end{proof}

\subsection{RLT/RCT Networks}

We now consider the special structure in the RLT and RCT networks. RLT and RCT networks are also common in applications, for example in the control of heating networks or viscoelastic systems. A lot more could be said about optimal control problems associated with these systems. This is because these systems have a realisation with an $A$ matrix that is symmetric and negative semi-definite, which endows the corresponding problems with a number of attractive properties. This has been exploited by several authors, see for example \cite{TG01,PS18}. In this subsection we content ourselves with the special case of \Cref{prob:1} in which the disturbance acts on the system state.

\begin{problem}\label{prob:3}
Given $\gamma{}>0$ and a state-space model
\begin{equation}\label{eq:ss3}
\begin{aligned}
\tfrac{d}{dt}x&=Ax+Bu+w,\,x\funof{0}=0,\\
y&=Cx+Du,\\
\end{aligned}
\end{equation}
find a stabilising control law
\begin{equation}\label{eq:ssc3}
\begin{aligned}
\tfrac{d}{dt}x_{\mathrm{K}}&=A_{\mathrm{K}}x_{\mathrm{K}}+B_{\mathrm{K}}y,\,x_{\mathrm{K}}\funof{0}=0,\\
u&=-C_{\mathrm{K}}x_{\mathrm{K}}-D_{\mathrm{K}}y,
\end{aligned}
\end{equation}
such that
\[
\norm{T_{zw}}_{H_\infty{}}<\gamma{},
\]
where $T_{zw}$ is the transfer function from $w$ to $z=\bm{y\\u}$ defined by \cref{eq:ss3,eq:ssc3}.
\end{problem}
Recall from \Cref{tab:1} that RLT networks have realisations with the special structure
\[
\mmfulla{A}{B}{C}{D}=\mmfulla{A}{\begin{array}{cc}B_1 & B_2\end{array}}{\begin{array}{c}B_1^{\mathsf{T}}\\-B_2^{\mathsf{T}}\end{array}}{\begin{array}{cc} D_{11} & D_{12} \\ -D_{12}^{\mathsf{T}} & D_{22}\end{array}},\;\zz{-A & B_2 \\ B_2^{\mathsf{T}} & D_{22}} \succeq 0, D_{11} \succeq{}0.
\]
The following theorem shows that if the process dynamics in \cref{eq:ss3} have this structure (under a slight strengthening), then \Cref{prob:3} admits an analytical solution. Again this control law may be dense, but can be synthesised in a simple and distributed manner as illustrated in \Cref{ex:3} below. A similar result holds for RCT networks, as explained in \Cref{rem:lct}, which follows the theorem statement.

\begin{theorem}\label{thm:5}
If $A,B,C$ and $D$ are as in entry {\bf{RLT}} of \Cref{tab:1}, where in addition
\begin{equation}\label{eq:extr}
\bm{
-A&B_2\\B_2^{\mathsf{T}}&D_{22}
}\succ{}0,
\end{equation}
then the following are equivalent:
\begin{enumerate}
  \item There exists a controller in the form of \cref{eq:ssc3} that solves \Cref{prob:3}.
  \item The controller in the form of \cref{eq:ssc3} with
  \[
  \mmfulla{A_{\mathrm{K}}}{B_{\mathrm{K}}}{C_{\mathrm{K}}}{D_{\mathrm{K}}}=
  \mmfulla{}{}{}{
  \begin{bmatrix}
  D_{11}&D_{12}\\-D_{12}^{\mathsf{T}}&D_{22}
  \end{bmatrix}
  -\begin{bmatrix}
  B_1^{\mathsf{T}}\\B_2^{\mathsf{T}}
  \end{bmatrix}A^{-1}
  \begin{bmatrix}B_1&-B_2
  \end{bmatrix}
  }
  \]
  solves \Cref{prob:3}.
\end{enumerate}
\end{theorem}

The proof of \Cref{thm:5} is given at the end of the subsection.

\begin{remark}\label{rem:lct}
An entirely analogous result to \Cref{thm:5} holds for RCT networks. In particular if $A,B,C$ and $D$ are as in entry {\bf{RCT}} of \Cref{tab:1}, with the strengthening
\[
\bm{
-A&B_1\\B_1^{\mathsf{T}}&D_{11}
}\succ{}0,
\]
then there is a controller satisfying \Cref{prob:3} if and only if the controller
\[
  \mmfulla{A_{\mathrm{K}}}{B_{\mathrm{K}}}{C_{\mathrm{K}}}{D_{\mathrm{K}}}=
  \mmfulla{}{}{}{
  \begin{bmatrix}
  D_{11}&D_{12}\\-D_{12}^{\mathsf{T}}&D_{22}
  \end{bmatrix}
  -\begin{bmatrix}
  B_1^{\mathsf{T}}\\B_2^{\mathsf{T}}
  \end{bmatrix}A^{-1}
  \begin{bmatrix}-B_1&B_2
  \end{bmatrix}
  }
  \]
satisfies \Cref{prob:3}.
\end{remark}

\begin{example}\label{ex:3}
The optimal control law in \Cref{thm:5} has a curious interpretation in terms of the electrical dual, which we will now investigate. This gives another example of an optimal controller where the electrical network perspective gives a distributed implementation even when the control law isn't sparse. Suppose we are given a planar RC network, such as the network illustrated in \Cref{fig:dual}, and would like to both synthesise and implement the optimal control law from \Cref{thm:5} using an electrical network. That is, we would like to optimally control the RC network by physically connecting another electrical network to the driving point terminal pairs. Just as in \Cref{ex:2}, we can synthesise the control law from \Cref{thm:5} using an electrical network. In this case the appropriate control law is obtained by replacing all the capacitors with open circuits. However connecting the two circuits together will not necessarily give the desired behaviour. 

This can be understood by observing that when describing an electrical network with an input-output system, if the input $u_k$ is a current, then the output $y_k$ is a voltage, and vice versa. Also for the feedback interconnection of two networks to make sense electrically, if the output $y_k$ of the network that we wish to control is a voltage, then the input to the network that synthesises the control law must also be a voltage (with the same matching for currents). However if we sythesise the controller network as described above, if the output of the process corresponds to a voltage, the input to the controller corresponds to a current. To correct this and obtain a control law that can be implemented through an electrical interconnection, we must sythesise the electrical dual of the controller network described above. For planar networks this can always be done without the use of transformers. The required steps are illustrated in \Cref{fig:dual1}. Note that if the network is not planar we may need to use transformers to synthesise the appropriate control law.

As a final remark we note that even if we follow the above procedure, unless each driving point forms a port, it will be necessary to interconnect the two networks using transformers. Transformless synthesis is an important area, particularly in the context of mechanical networks \cite{Smi02}, and the above issue has connections to a number of unsolved problems. We point the interested reader to \cite{Wil10a,Che15}, as well as to \cite[Problem 2]{Hug20} and \cite[\S{21.8}]{HMS18}.
\end{example}

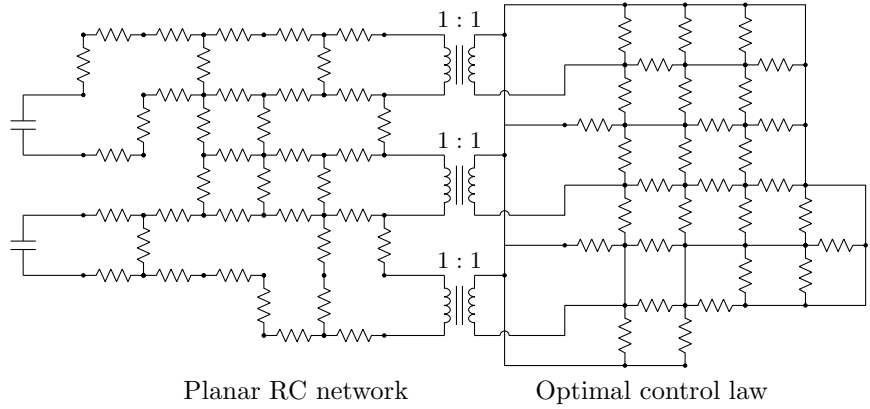
\begin{figure}
\centering
    \ctikzset{bipoles/length=.56cm}
    \ctikzset{bipoles/thickness=1}
\begin{tikzpicture}[font=\normalsize,scale=.4]
    \draw (-14,4) to [C](-14,2);
    \draw (-14,8) to [C](-14,6);
    \draw (0,2) to [L] (0,0);
    \draw (0.4,0.35) to (0.4,1.65);
    \draw (0.6,0.35) to (0.6,1.65);
    \draw (1,0) to [L] (1,2);
    \draw (0.5,2.5) node {\small $1:1$};
    \draw (0,6) to [L] (0,4);
    \draw (0.4,4.35) to (0.4,5.65);
    \draw (0.6,4.35) to (0.6,5.65);
    \draw (1,4) to [L] (1,6);
    \draw (0.5,6.5) node {\small $1:1$};
    \draw (0,10) to [L] (0,8);
    \draw (0.4,8.35) to (0.4,9.65);
    \draw (0.6,8.35) to (0.6,9.65);
    \draw (1,8) to [L] (1,10);
    \draw (0.5,10.5) node {\small $1:1$};
    \draw (-14,4) to [short] (-12,4);
    \draw (-12,2) to [short] (-14,2);
    \draw (-14,8) to [short] (-12,8);
    \draw (-12,6) to [short] (-14,6);
    \draw (-2,0) to [short] (0,0);
    \draw (-2,2) to [short] (0,2);
    \draw (-2,4) to [short] (0,4);
    \draw (-2,6) to [short] (0,6);
    \draw (-2,8) to [short] (0,8);
    \draw (-2,10) to [short] (0,10);
    \draw (-12,2) to [R,*-*] (-10,2);
    \draw (-10,2) to [R,*-*] (-8,2) to [R,*-*] (-6,2) to [R,*-*](-6,0) to [R,*-*](-4,0) to [R,*-*] (-2,0);
    \draw (-12,4) to[R,*-*](-10,4);
    \draw (-10,4) to[R,*-*](-8,4) to[R,*-*](-6,4) to[R,*-*] (-4,4) to[R,*-*] (-2,4) to[R,*-*](-2,2);
    \draw (-12,6) to[R,*-*] (-10,6) to[R,*-*] (-10,8) to[R,*-*] (-8,8);
    \draw (-8,8) to[R,*-*] (-6,8) to[R,*-*] (-4,8) to[R,*-*] (-2,8) to[R,*-*] (-2,6);
    \draw (-12,8) to[R,*-*] (-12,10) to[R,*-*] (-10,10) to[R,*-*] (-8,10);
    \draw (-8,10) to[R,*-*] (-6,10) to[R,*-*] (-4,10) to[R,*-*] (-2,10);
    \draw (-8,4) to[R,*-*] (-8,6) to[R,*-*] (-8,8) to[R,*-*] (-8,10);
    \draw (-8,6) to[R,*-*] (-6,6) to[R,*-*] (-4,6) to[R,*-*] (-2,6);
    \draw (-4,0) to[R,*-*] (-4,2) to[R,*-*] (-4,4) to[R,*-*] (-4,6);
    \draw (-10,2) to [R,*-*] (-10,4);
    \draw (-6,4) to [R,*-*] (-6,6) to [R,*-*] (-6,8);
    \draw (-4,8) to [R,*-*] (-4,10);
    
    \draw[black] (3+3,-1) to [R,*-*,color=black] (3+3,1);
    \draw[black] (3+3,3) to[R,*-*,color=black] (3+3,5) to[R,*-*,color=black] (3+3,7) to[R,*-*,color=black] (3+3,9) to[R,*-*,color=black] (3+3,11);
    \draw[black] (5+3,-1) to [R,*-*,color=black] (5+3,1);
    \draw[black] (5+3,3) to[R,*-*,color=black] (5+3,5) to[R,*-*,color=black] (5+3,7) to[R,*-*,color=black] (5+3,9) to[R,*-*,color=black] (5+3,11);
    \draw[black] (7+3,1) to[R,*-*,color=black] (7+3,3) to[R,*-*,color=black] (7+3,5) to[R,*-*,color=black] (7+3,7) to[R,*-*,color=black] (7+3,9) to[R,*-*,color=black] (7+3,11);
    \draw[black] (9+3,1) to[R,*-*,color=black] (9+3,3) to[R,*-*,color=black] (9+3,5);
    \draw[black] (7+3,1) to[R,*-*,color=black] (5+3,1) to[R,*-*,color=black](3+3,1);  
    \draw[black] (11+3,3) to[R,*-*,color=black] (9+3,3);
    \draw[black] (5+3,3) to[R,*-*,color=black] (3+3,3)to[R,*-*,color=black] (1+3,3);
    \draw[black] (9+3,5) to[R,*-*,color=black] (7+3,5) to[R,*-*,color=black](5+3,5) to[R,*-*,color=black] (3+3,5);  
    \draw[black] (9+3,7) to[R,*-*,color=black] (7+3,7) to[R,*-*,color=black](5+3,7);
    \draw[black] (3+3,7) to[R,*-*,color=black] (1+3,7);
    \draw[black] (9+3,9) to[R,*-*,color=black] (7+3,9);
    \draw[black] (5+3,9) to[R,*-*,color=black] (3+3,9);
    \draw[black] (5+3,-1) to (-1+3,-1) to (-1+3,11) to (9+3,11) to (9+3,5) to (11+3,5) to (11+3,1) to (7+3,1);
    \draw[black] (1+3,3) to (-1+3,3);
    \draw[black] (1+3,7) to (-1+3,7);
    \draw[black] (9+3,3) to (5+3,3) to (5+3,1); 
    \draw[black] (3+3,3) to (3+3,1) to (1+3,1);
    \draw[black] (3+3,5) to (1+3,5);
    \draw[black] (5+3,7) to (3+3,7);
    \draw[black] (7+3,9) to (5+3,9);
    \draw[black] (3+3,9) to (1+3,9);
    \draw (1,0) to (1.85,0);
    \draw (1.85,0) arc (180:0:.15);
    \draw (2.15,0) to(4,0) to (4,1);
    \draw (1,2) to [short,-*] (2,2);
    \draw (1,4) to (1.85,4);
    \draw (1.85,4) arc (180:0:.15);
    \draw (2.15,4) to(4,4) to (4,5);
    \draw (1,6) to [short,-*] (2,6);
    \draw (1,8) to (1.85,8);
    \draw (1.85,8) arc (180:0:.15);
    \draw (2.15,8) to(4,8) to (4,9);
    \draw (1,10) to [short,-*] (2,10);

    \end{tikzpicture}

    {$\;\quad\quad\,$ Planar RC network $\,\qquad\qquad$ Optimal control law}

    \caption{\label{fig:dual}Illustration of the controller implementation in \Cref{ex:3}. The objective is to implement the optimal control law from \Cref{thm:5} for the given RC network (network constructed out of resistors and capacitors) electrically. To do this, the terminals across the \emph{k}-th driving point in the RC network are connected to the terminals across the \emph{k}-th driving point in the optimal control law network using a transformer with turns ratio $1:1$. The optimal control law is the electrical dual (see e.g. \cite{Gui53}) of the original circuit with all capacitors open circuited, and is constructed according to {\footnotesize \textsc{({\tiny A})}}--{\footnotesize \textsc{({\tiny C})}} in \Cref{fig:dual1}.}
\end{figure}

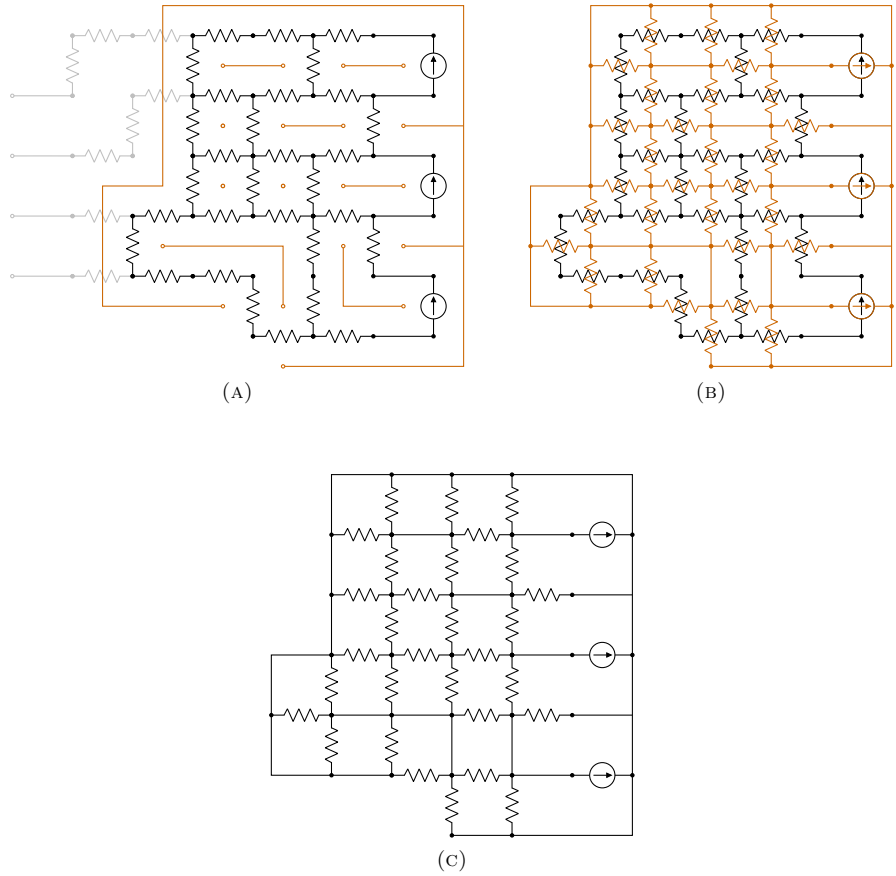
\begin{figure} 
\centering   
\ctikzset{bipoles/thickness=1}
\ctikzset{bipoles/length=1.4cm}
\subfloat[]{
  \begin{tikzpicture}[font=\normalsize,scale=.4,transform shape]
    \draw (0,0) to [I,*-*](0,2);
    \draw (0,4) to [I,*-*](0,6);
    \draw (0,8) to [I,*-*](0,10);
    \draw[lightgray] (-14,4) to [short,o-*,color=lightgray] (-12,4);
    \draw[lightgray] (-12,2) to [short,*-o,color=lightgray] (-14,2);
    \draw[lightgray] (-14,8) to [short,o-*,color=lightgray] (-12,8);
    \draw[lightgray] (-12,6) to [short,*-o,color=lightgray] (-14,6);
    \draw (-2,0) to [short] (0,0);
    \draw (-2,2) to [short] (0,2);
    \draw (-2,4) to [short] (0,4);
    \draw (-2,6) to [short] (0,6);
    \draw (-2,8) to [short] (0,8);
    \draw (-2,10) to [short] (0,10);
    \draw[lightgray] (-12,2) to [R,*-*,color=lightgray] (-10,2);
    \draw (-10,2) to [R,*-*] (-8,2) to [R,*-*] (-6,2) to [R,*-*](-6,0) to [R,*-*](-4,0) to [R,*-*] (-2,0);
    \draw[lightgray] (-12,4) to[R,*-*,color=lightgray](-10,4);
    \draw (-10,4) to[R,*-*](-8,4) to[R,*-*](-6,4) to[R,*-*] (-4,4) to[R,*-*] (-2,4) to[R,*-*](-2,2);
    \draw[lightgray] (-12,6) to[R,*-*,color=lightgray] (-10,6) to[R,*-*,color=lightgray] (-10,8) to[R,*-*,color=lightgray] (-8,8);
    \draw (-8,8) to[R,*-*] (-6,8) to[R,*-*] (-4,8) to[R,*-*] (-2,8) to[R,*-*] (-2,6);
    \draw[lightgray] (-12,8) to[R,*-*,color=lightgray] (-12,10) to[R,*-*,color=lightgray] (-10,10) to[R,*-*,color=lightgray] (-8,10);
    \draw (-8,10) to[R,*-*] (-6,10) to[R,*-*] (-4,10) to[R,*-*] (-2,10);
    \draw (-8,4) to[R,*-*] (-8,6) to[R,*-*] (-8,8) to[R,*-*] (-8,10);
    \draw (-8,6) to[R,*-*] (-6,6) to[R,*-*] (-4,6) to[R,*-*] (-2,6);
    \draw (-4,0) to[R,*-*] (-4,2) to[R,*-*] (-4,4) to[R,*-*] (-4,6);
    \draw (-10,2) to [R,*-*] (-10,4);
    \draw (-6,4) to [R,*-*] (-6,6) to [R,*-*] (-6,8);
    \draw (-4,8) to [R,*-*] (-4,10);
    \draw[orange!80!black] (-7,7) to[short,o-o,color=orange!80!black] (-7,7);
    \draw[orange!80!black] (-7,5) to[short,o-o,color=orange!80!black] (-7,5);
    \draw[orange!80!black] (-5,5) to[short,o-o,color=orange!80!black] (-5,5);
    \draw[orange!80!black] (-5,-1) to [short,o-,color=orange!80!black](1,-1) to (1,11) to (-9,11) to (-9,5) to (-11,5) to (-11,1) to [short,color=orange!80!black,-o] (-7,1);
    \draw[orange!80!black] (-1,3) to [short,o-,color=orange!80!black] (1,3);
    \draw[orange!80!black] (-1,7) to[short,o-,color=orange!80!black] (1,7);
    \draw[orange!80!black] (-9,3) to  [short,o-,color=orange!80!black](-5,3) to [short,-o,color=orange!80!black] (-5,1);  
    \draw[orange!80!black] (-3,3) [short,o-,color=orange!80!black] to (-3,1) to [short,-o,color=orange!80!black] (-1,1);
    \draw[orange!80!black] (-3,5) to[short,o-o,color=orange!80!black] (-1,5);
    \draw[orange!80!black] (-5,7) to [short,o-o,color=orange!80!black](-3,7);
    \draw[orange!80!black] (-7,9) to [short,o-o,color=orange!80!black](-5,9);
    \draw[orange!80!black] (-3,9) to[short,o-o,color=orange!80!black] (-1,9);
  \end{tikzpicture}
}\hspace{.5cm}
\subfloat[]{  
  \begin{tikzpicture}[font=\normalsize,scale=.4,transform shape]
    \draw (0,0) to [I,*-*](0,2);
    \draw (0,4) to [I,*-*](0,6);
    \draw (0,8) to [I,*-*](0,10);
    
    \draw (-2,0) to [short] (0,0);
    \draw (-2,2) to [short] (0,2);
    \draw (-2,4) to [short] (0,4);
    \draw (-2,6) to [short] (0,6);
    \draw (-2,8) to [short] (0,8);
    \draw (-2,10) to [short] (0,10);
    \draw (-10,2) to [R,*-*] (-8,2) to [R,*-*] (-6,2) to [R,*-*](-6,0) to [R,*-*](-4,0) to [R,*-*] (-2,0);
    \draw (-10,4) to[R,*-*](-8,4) to[R,*-*](-6,4) to[R,*-*] (-4,4) to[R,*-*] (-2,4) to[R,*-*](-2,2);
    \draw (-8,8) to[R,*-*] (-6,8) to[R,*-*] (-4,8) to[R,*-*] (-2,8) to[R,*-*] (-2,6);
    \draw (-8,10) to[R,*-*] (-6,10) to[R,*-*] (-4,10) to[R,*-*] (-2,10);
    \draw (-8,4) to[R,*-*] (-8,6) to[R,*-*] (-8,8) to[R,*-*] (-8,10);
    \draw (-8,6) to[R,*-*] (-6,6) to[R,*-*] (-4,6) to[R,*-*] (-2,6);
    \draw (-4,0) to[R,*-*] (-4,2) to[R,*-*] (-4,4) to[R,*-*] (-4,6);
    \draw (-10,2) to [R,*-*] (-10,4);
    \draw (-6,4) to [R,*-*] (-6,6) to [R,*-*] (-6,8);
    \draw (-4,8) to [R,*-*] (-4,10);
    \draw[orange!80!black](-1,1) to [I,*-*,color=orange!80!black] (1,1);
    \draw[orange!80!black](-1,5) to [I,*-*,color=orange!80!black] (1,5);
    \draw[orange!80!black](-1,9) to [I,*-*,color=orange!80!black] (1,9);
    \draw[orange!80!black] (-3,-1) to [R,*-*,color=orange!80!black] (-3,1);
    \draw[orange!80!black] (-3,3) to[R,*-*,color=orange!80!black] (-3,5) to[R,*-*,color=orange!80!black] (-3,7) to[R,*-*,color=orange!80!black] (-3,9) to[R,*-*,color=orange!80!black] (-3,11);
    \draw[orange!80!black] (-5,-1) to [R,*-*,color=orange!80!black] (-5,1);
    \draw[orange!80!black] (-5,3) to[R,*-*,color=orange!80!black] (-5,5) to[R,*-*,color=orange!80!black] (-5,7) to[R,*-*,color=orange!80!black] (-5,9) to[R,*-*,color=orange!80!black] (-5,11);
    \draw[orange!80!black] (-7,1) to[R,*-*,color=orange!80!black] (-7,3) to[R,*-*,color=orange!80!black] (-7,5) to[R,*-*,color=orange!80!black] (-7,7) to[R,*-*,color=orange!80!black] (-7,9) to[R,*-*,color=orange!80!black] (-7,11);
    \draw[orange!80!black] (-9,1) to[R,*-*,color=orange!80!black] (-9,3) to[R,*-*,color=orange!80!black] (-9,5);
    \draw[orange!80!black] (-7,1) to[R,*-*,color=orange!80!black] (-5,1) to[R,*-*,color=orange!80!black](-3,1); 
    \draw[orange!80!black] (-11,3) to[R,*-*,color=orange!80!black] (-9,3);
    \draw[orange!80!black] (-5,3) to[R,*-*,color=orange!80!black] (-3,3)to[R,*-*,color=orange!80!black] (-1,3);
    \draw[orange!80!black] (-9,5) to[R,*-*,color=orange!80!black] (-7,5) to[R,*-*,color=orange!80!black](-5,5) to[R,*-*,color=orange!80!black] (-3,5);  
    \draw[orange!80!black] (-9,7) to[R,*-*,color=orange!80!black] (-7,7) to[R,*-*,color=orange!80!black](-5,7);
    \draw[orange!80!black] (-3,7) to[R,*-*,color=orange!80!black] (-1,7);
    \draw[orange!80!black] (-9,9) to[R,*-*,color=orange!80!black] (-7,9);
    \draw[orange!80!black] (-5,9) to[R,*-*,color=orange!80!black] (-3,9);
    \draw[orange!80!black] (-5,-1) to (1,-1) to (1,11) to (-9,11) to (-9,5) to (-11,5) to (-11,1) to (-7,1);
    \draw[orange!80!black] (-1,3) to (1,3);
    \draw[orange!80!black] (-1,7) to (1,7);
    \draw[orange!80!black] (-9,3) to (-5,3) to (-5,1);  
    \draw[orange!80!black] (-3,3) to (-3,1) to (-1,1);
    \draw[orange!80!black] (-3,5) to (-1,5);
    \draw[orange!80!black] (-5,7) to (-3,7);
    \draw[orange!80!black] (-7,9) to (-5,9);
    \draw[orange!80!black] (-3,9) to (-1,9);
  \end{tikzpicture}
    }

\vspace{.5cm}
\subfloat[]{   
  \begin{tikzpicture}[font=\normalsize,scale=.4,transform shape]
    \draw[black](-1,1) to [I,*-*,color=black] (1,1);
    \draw[black](-1,5) to [I,*-*,color=black] (1,5);
    \draw[black](-1,9) to [I,*-*,color=black] (1,9);
    \draw[black] (-3,-1) to [R,*-*,color=black] (-3,1);
    \draw[black] (-3,3) to[R,*-*,color=black] (-3,5) to[R,*-*,color=black] (-3,7) to[R,*-*,color=black] (-3,9) to[R,*-*,color=black] (-3,11);
    \draw[black] (-5,-1) to [R,*-*,color=black] (-5,1);
    \draw[black] (-5,3) to[R,*-*,color=black] (-5,5) to[R,*-*,color=black] (-5,7) to[R,*-*,color=black] (-5,9) to[R,*-*,color=black] (-5,11);
    \draw[black] (-7,1) to[R,*-*,color=black] (-7,3) to[R,*-*,color=black] (-7,5) to[R,*-*,color=black] (-7,7) to[R,*-*,color=black] (-7,9) to[R,*-*,color=black] (-7,11);
    \draw[black] (-9,1) to[R,*-*,color=black] (-9,3) to[R,*-*,color=black] (-9,5);
    \draw[black] (-7,1) to[R,*-*,color=black] (-5,1) to[R,*-*,color=black](-3,1); 
    \draw[black] (-11,3) to[R,*-*,color=black] (-9,3);
    \draw[black] (-5,3) to[R,*-*,color=black] (-3,3)to[R,*-*,color=black] (-1,3);
    \draw[black] (-9,5) to[R,*-*,color=black] (-7,5) to[R,*-*,color=black](-5,5) to[R,*-*,color=black] (-3,5);  
    \draw[black] (-9,7) to[R,*-*,color=black] (-7,7) to[R,*-*,color=black](-5,7);
    \draw[black] (-3,7) to[R,*-*,color=black] (-1,7);
    \draw[black] (-9,9) to[R,*-*,color=black] (-7,9);
    \draw[black] (-5,9) to[R,*-*,color=black] (-3,9);
    \draw[black] (-5,-1) to (1,-1) to (1,11) to (-9,11) to (-9,5) to (-11,5) to (-11,1) to (-7,1);
    \draw[black] (-1,3) to (1,3);
    \draw[black] (-1,7) to (1,7);
    \draw[black] (-9,3) to (-5,3) to (-5,1);  
    \draw[black] (-3,3) to (-3,1) to (-1,1);
    \draw[black] (-3,5) to (-1,5);
    \draw[black] (-5,7) to (-3,7);
    \draw[black] (-7,9) to (-5,9);
    \draw[black] (-3,9) to (-1,9);
  \end{tikzpicture}
}
\caption{\label{fig:dual1}Construction of the optimal control law for the network from \Cref{fig:dual}. {\footnotesize \textsc{({\tiny A})}} First draw the RC network with the capacitors removed. This may leave some resistor branches hanging. These are shaded light grey, and can be ignored for the rest of this process. Since the remaining network is planar, its graph partitions the plane into regions. Place a node in each region, including the region exterior to the circuit, as illustrated in orange. To simplify the drawing of the next step, multiple nodes connected by short circuits are drawn within each region (this is equivalent to a single node). {\footnotesize \textsc{({\tiny B})}} Each resistor in the original circuit lies between two regions. For each resistor, draw in a new resistor that connects these regions. Set the resistance of the new resistor to be the inverse of that of the original resistor. Follow the same procedure for the driving points, adding in a new driving point connecting the regions. {\footnotesize \textsc{({\tiny C})}} The electrical dual.}
\end{figure}

We now give the proof of \Cref{thm:5}.

\begin{proof}
\textit{(2)} $\Rightarrow{}$\textit{(1)}: Immediate.

\textit{(1)} $\Rightarrow{}$\textit{(2)}: The proof will consist of two steps. In step 1 we will use a least squares argument to find a lower bound on the achievable $\gamma$ (that is a number $\gamma^*$ such that the \Cref{prob:3} is solvable only if $\gamma>\gamma^*$). In step 2 we will show that the controller in the theorem statement achieves $\norm{T_{zw}}_\infty=\gamma^*$.

Step 1: First note the following simple lower bound on the $H_\infty{}$ norm of a (stable) transfer function
\[
\norm{G}_\infty{}=\sup_{\omega\in\R}\norm{G\jw}_2\geq{}\norm{G\funof{0}}_2.
\]
Next note that
\[
T_{zw}\s=\bm{I\\-K\s}\funof{I+G\s{}K\s}^{-1}C\funof{sI-A}^{-1},
\]
where $K\s=C_{\mathrm{K}}\funof{sI-A_{\mathrm{K}}}^{-1}B_{\mathrm{K}}+D_{\mathrm{K}}$. We can thus find a lower bound on the achievable $H_\infty{}$ performance (as suggested in \cite{BPR20}) by solving the static optimisation problem
\[
\min\cfunof{\norm{\bm{I\\-\tilde{K}}\funof{I+G\funof{0}\tilde{K}}^{-1}CA^{-1}}_2:\tilde{K}\text{ has entries in }\R}.
\]
A least squares argument (see \cite[Lemma 1]{PBR19}) shows that $\tilde{K}=G\funof{0}^{\mathsf{T}}$ minimises the above, and so $\norm{T_{zw}}_\infty{}\geq{}\gamma^*$, where
\[
\gamma^*=\norm{\bm{I\\-G\funof{0}^{\mathsf{T}}}\funof{I+G\funof{0}G\funof{0}^{\mathsf{T}}}^{-1}CA^{-1}}_2.
\]

Step 2: We will now show that $K\s=G\funof{0}^{\mathsf{T}}$ achieves the lower bound from step 1. First it is convenient to rewrite \cref{eq:ss3} in chain form, in which $u$ and $y$ are split in line with the partitioning of the matrices in the {\bf{RLT}} entry of \Cref{tab:1}:
\[
\begin{aligned}
\tfrac{d}{dt}x&=\bar{A}x+w+\bar{B}_1u_1+\bar{B}_2y_2,\\
y_1&=\bar{C}_1u_1+\bar{D}_{11}u_1+\bar{D}_{12}y_2,\\
u_2&=\bar{C}_2y_2+\bar{D}_{21}u_1+\bar{D}_{22}y_2.
\end{aligned}
\]
Going through the necessary algebra we find that
\[
\scalemath{.8}{
  \sqfunof{\begin{array}{c;{2pt/2pt}cc}
\bar{A}&\bar{B}_1&\bar{B}_2\\\hdashline[2pt/2pt]
\bar{C}_1&\bar{D}_{11}&\bar{D}_{12}\\
\bar{C}_2&\bar{D}_{21}&\bar{D}_{22}
  \end{array}
  }}
  =
 \scalemath{.8}{
  \sqfunof{\begin{array}{c;{2pt/2pt}cc}
A+B_2D_{22}^{-1}B_2^{\mathsf{T}}&B_1+B_2D_{22}^{-1}D_{12}^{\mathsf{T}}&B_2D_{22}^{-1}\\\hdashline[2pt/2pt]
B_1^{\mathsf{T}}+D_{12}D_{22}^{-1}B_2^{\mathsf{T}}&D_{11}+D_{12}D_{22}^{-1}D_{12}^{\mathsf{T}}&D_{12}D_{22}^{-1}\\
D_{22}^{-1}B_2^{\mathsf{T}}&D_{22}^{-1}D_{12}^{\mathsf{T}}&D_{22}^{-1}
  \end{array}
  }}.
\]
Observe in particular that $\bar{A}\prec{}0,\bar{B}=\bar{C}^{\mathsf{T}}$ and $\bar{D}\succeq{}0$ (these follow from the conditions in the {\bf{RLT}} entry in \Cref{tab:1}, with the strengthening in \cref{eq:extr}). We can similarly rewrite the control law $u=-G\funof{0}^{\mathsf{T}}{}y$ as
\[
\bm{u_1\\y_2}=-\underbrace{\funof{\bar{D}-\bar{B}^{\mathsf{T}}\bar{A}^{-1}\bar{B}}}_{\eqqcolon{}\bar{K}}\bm{y_1\\u_2},
\]
and note that $\bar{K}\succeq{}0$. We can then rewrite $T_{zw}$ in terms of these new variables as
\[
\begin{aligned}
PT_{zw}&=\bm{I\\-\bar{K}}\funof{I+\bar{G}\s\bar{K}}^{-1}\bar{C}\funof{sI-\bar{A}}^{-1},\\
&=\bm{I\\-\bar{K}}\funof{I+\bar{D}\bar{K}+\bar{B}^{\mathsf{T}}\funof{sI-\bar{A}}^{-1}\bar{B}\bar{K}}\bar{B}^{\mathsf{T}}\funof{sI-\bar{A}}^{-1},
\end{aligned}
\]
where $P$ is a permutation matrix. Applying the Woodbury matrix identity and rearranging then shows that the above equals
\[
\begin{aligned}
\bm{I\\-\bar{K}}\funof{I+\bar{D}\bar{K}}^{-1}\bar{B}^{\mathsf{T}}\funof{sI-\bar{A}+\bar{B}\bar{K}\funof{I+\bar{D}\bar{K}}^{-1}\bar{B}^{\mathsf{T}}}^{-1}.
\end{aligned}
\]
Note we can be sure that \funof{I+\bar{D}\bar{K}} is invertible since $\bar{D}\succeq{}0$ and $\bar{K}\succeq{}0$, and therefore $\det\funof{\lambda{}I+\bar{D}\bar{K}}$ is zero only if $\lambda\leq{}0$. Now let
\[
\begin{aligned}
X&=-\bar{A}+\bar{B}\bar{K}\funof{I+\bar{D}\bar{K}}^{-1}\bar{B}^{\mathsf{T}},\\
&=-\bar{A}+\bar{B}\bar{K}^{\frac{1}{2}}\funof{I+\bar{K}^{\frac{1}{2}}\bar{D}\bar{K}^{\frac{1}{2}}}^{-1}\bar{K}^{\frac{1}{2}}\bar{B}^{\mathsf{T}},
\end{aligned}.
\]
from which we see that $X\succ{}0$. Since from step 1, $\norm{PT_{zw}\funof{0}}_2=\gamma^*$,
\[
\norm{\bm{I\\-\bar{K}}\funof{I+\bar{D}\bar{K}}^{-1}\bar{B}^{\mathsf{T}}X^{-1}}_2=\gamma^*,
\]
which implies that
\[
\begin{aligned}
\norm{T_{zw}}_\infty{}&=\norm{\bm{I\\-\bar{K}}\funof{I+\bar{D}\bar{K}}^{-1}\bar{B}^{\mathsf{T}}X^{-1}\funof{sX^{-1}+I}^{-1}}_\infty{},\\
&\leq{}\gamma^*\norm{\funof{sX^{-1}+I}^{-1}}_\infty{}=\gamma^*
\end{aligned}
\]
as required.
\end{proof}

\section{Conclusions}

The main contributions of the paper are to show that:
\begin{enumerate}
  \item the external behaviors of electrical networks constructed out of subsets of the reciprocal elements (resistors, inductors, capacitors and transformers) admit observable and controllable state-space realisations with highly structured $A$, $B$, $C$ and $D$ matrices;
  \item the structure in these matrices can be exploited to simplify, or solve analytically, the Riccati equations and \glspl{lmi} arising in a range of $H_2$ and $H_{\infty}$ optimal control problems.
\end{enumerate}
This reveals that some optimal controllers for systems that can be modelled by reciprocal networks inherit the structural properties of the network that describes the system. This can be used to give distributed or decentralised controller implementations, and the results are illustrated on examples motivated by constrained least squares problems, electrical power systems, consensus algorithms and heating networks.

\bibliographystyle{amsplain}
\bibliography{references.bib}

\end{document}